\documentclass[12pt,leqno]{article}
\usepackage{amssymb,amsfonts,amsmath,amsthm,amscd,bbold,mathrsfs}
\usepackage{MnSymbol}
\usepackage{color}
\usepackage{psfrag}
\usepackage{graphicx}
\usepackage{centernot}
\usepackage{stmaryrd}
\def\CC{\leavevmode\setbox0=\hbox{h}\dimen=\ht0 \advance \dimen by-1ex\rlap{\raise.9\dimen\hbox{\kern .15em \char'27}} C}

\def\IE{{\mathbb E}}

\def\IP{{\mathbb P}}
\def\IR{{\mathbb R}}
\def\IN{{\mathbb N}}

\def\IZ{{\mathbb Z}}

\def\n{\noindent}

\def\dis{\displaystyle}
\def\o{\omega}
\def\fr{\mbox{\footnotesize $\dis\frac{1}{2}$}}

\def\ov{\overline}
\def\ve{\varepsilon}
\def\f{\footnotesize}

\def\r{\rightarrow}

\def\wh{\widehat}
\def\wt{\widetilde}

\def\cA{{\cal A}}
\def\cB{{\cal B}}
\def\cC{{\cal C}}
\def\cD{{\cal D}}

\def\cF{{\cal F}}

\def\cG{{\cal G}}

\def\cV{{\cal V}}

\def\lo{\lambda_{1,\omega}}
\def\llo{\lambda_{2,\omega}}
\def\lob{\lambda_{1,\omega}(B_\ell)}
\def\llob{\lambda_{2,\omega}(B_\ell)}
\def\lod{\lambda_{1,\omega}(D_0)}

\def\lol{(\log \ell)}

\dimendef\dimen=0

\newtheorem{theorem}{Theorem}[section]
\newtheorem{lemma}[theorem]{Lemma}

\newtheorem{proposition}[theorem]{Proposition}
\newtheorem{remark}[theorem]{Remark}

\begin{document}

\noindent
~

\bigskip
\begin{center}
{\bf ON THE SPECTRAL GAP IN THE KAC-LUTTINGER MODEL \\ AND BOSE-EINSTEIN CONDENSATION}
\end{center}

\begin{center}
Alain-Sol Sznitman

\bigskip
\textit{In memory of Francis Comets}
\end{center}

\begin{center}
%Preliminary Draft
\end{center}

\begin{abstract}
We consider the Dirichlet eigenvalues of the Laplacian among a Poissonian cloud of hard spherical obstacles of fixed radius in large boxes of $\IR^d$, $d \ge 2$. In a large box of side-length $2 \ell$ centered at the origin, the lowest eigenvalue is known to be typically of order $(\log \ell)^{-2/d}$. We show here that with probability arbitrarily close to $1$ as $\ell$ goes to infinity, the spectral gap stays bigger than $\sigma (\log \ell)^{-(1 + 2/d)}$, where the small positive number $\sigma$ depends on how close to $1$ one wishes the probability. Incidentally, the scale $(\log \ell)^{-(1+ 2/d)}$ is expected to capture the correct size of the gap. Our result involves the proof of new deconcentration estimates. Combining this lower bound on the spectral gap with the results of Kerner-Pechmann-Spitzer, we infer a type-I generalized Bose-Einstein condensation in probability for a Kac-Luttinger system of non-interacting bosons among Poissonian spherical impurities, with the sole macroscopic occupation of the one-particle ground state when the density exceeds the critical value.
\end{abstract}

\vspace{4cm}

\n
Departement Mathematik \\
ETH Z\"urich\\
CH-8092 Z\"urich\\
Switzerland

\newpage
\thispagestyle{empty}
~

\newpage
\setcounter{page}{1}

\setcounter{section}{-1}
\section{Introduction}

In this article we are interested in the Kac-Luttinger model   \cite{KacLutt73},  \cite{KacLutt74}, and consider the Dirichlet eigenvalues of the Laplacian among a Poissonian cloud of hard spherical obstacles in large boxes of $\IR^d$, $d \ge 2$. Our main result states an asymptotic lower bound for the spectral gap, that is, for the difference between the second lowest and the lowest Dirichlet eigenvalues. In a large box of side-length $2 \ell$ centered at the origin, the lowest eigenvalue is known to be typically of order $(\log \ell)^{-2/d}$. We show here that with probability arbitrarily close to $1$ as $\ell$ goes to infinity, the spectral gap stays bigger than $\sigma (\log \ell)^{-(1 + 2/d)}$, where the small positive number $\sigma$ depends on how close to $1$ one wishes the probability. Incidentally, $(\log \ell)^{-(1 + 2/d)}$ is expected to capture the correct size of the spectral gap. Whereas detailed information on the statistics of eigenvalues in large boxes is known for various kinds of random potentials, see \cite{BiskKoni16}, \cite{GermKlop13}, Chapter 6 \S 3 of \cite{Koni16}, or Section 6 of \cite{Astr16}, much less seems to be known in the case of hard or soft Poissonian obstacles. Part of the difficulty stems from the delicate nature of the competition between the so-called  ``clearings''. Loosely speaking, these are nearly spherical pockets of rough size $(\log \ell)^{1/d}$ with a rarefied presence of the obstacles that underpin low eigenvalues. Various facets of these questions emerge in related studies concerning Brownian motion in a Poissonian potential and kindred models, see \cite{Szni98a}, \cite{Koni16}, \cite{Astr16}, and references therein, as well as \cite{DingXu20}, \cite{DingFukuSunXu21}, and \cite{PoisSime22} for some recent developments. In the present work we bring new deconcentration estimates. We also have a different purpose. Combining the lower bound on the spectral gap, the known Lifshitz tail behavior of the model, and the results of Kerner, Pechmann and Spitzer in \cite{KernPechSpit20}, we infer a so-called type-I generalized Bose-Einstein condensation in probability for a Kac-Luttinger system of non-interacting bosons among Poissonian spherical impurities in the thermodynamic limit, which appears to be novel, see \cite{LenoPastZagr04}, \cite{KernPechSpit20}, \cite{Pech20}.

\medskip
We now describe the results in more detail, and refer to Section 1 for additional notation and references. We consider $\IR^d$, $d \ge 2$, and the canonical law $\IP$ on the space $\Omega$ of locally finite simple point measures on $\IR^d$, of a Poisson cloud of constant intensity $\nu > 0$. To each point of the cloud we attach a spherical obstacle corresponding to a closed ball of radius $a > 0$ centered at the point. Given $\ell > 0$ and the open box of side-length $2 \ell$ centered at the origin $B_\ell = (-\ell, \ell)^d$, we denote by $B_{\ell,\o}$ the complement of the obstacle set in $B_\ell$. We write
\begin{equation}\label{0.1}
0 < \lambda_{1,\o} (B_\ell) \le \lambda_{2,\o} (B_\ell) \le \dots \le \lambda_{i,\o} (B_\ell) \le \dots \le \infty,
\end{equation}
for the successive Dirichlet eigenvalues of $-\frac{1}{2}\, \Delta$ in $B_{\ell,\o}$, repeated according to their multiplicity, with the convention that $\lambda_{i,\o}(B_\ell) = \infty$ for all $i \ge 1$, when $B_{\ell,\o} = \emptyset$ (otherwise all $\lambda_{i,\o}(B_\ell)$ are finite).

\medskip
We let $\o_d$ stand for the volume of an open ball of radius $1$, $\lambda_d$ for the corresponding principal Dirichlet eigenvalue of $-\frac{1}{2} \,\Delta$, and define
\begin{align}
R_0  &= \Big(\frac{d}{\nu \o_d}\Big)^{1/d},\;\mbox{the radius of a ball of volume $\frac{d}{\nu}$}, \label{0.2}
\\[1ex]
c_0 & = \lambda_d \Big(\frac{d}{\nu \o_d}\Big)^{-2/d},  \; \mbox{the principal Dirichlet eigenvalues of $-\frac{1}{2} \,\Delta$ in an open ball} \label{0.3}
\\[-1.5ex]
&\hspace{2.5cm}\,\mbox{ of radius $R_0$} .\nonumber
 \end{align}

\n
One knows from Theorem 4.6, p.~191 of \cite{Szni98a} (which treats the more delicate case of soft obstacles, see (\ref{1.18}), but remains valid in the present context) that
\begin{align}
&\mbox{there exists $\chi \in (0,d)$ and $\gamma > 0$ such that on a set of full $\IP$-measure, for large $\ell$}\label{0.4}
\\
&c_0 (\log \ell)^{-2/d} - (\log \ell)^{-(2 + \chi)/d} \le \lambda_{1,\o} (B_\ell) \le c_0 (\log \ell)^{-2/d} + \gamma (\log \ell)^{-3/d}. \nonumber
\end{align}

\n
The main result of the present article is Theorem \ref{theo6.1}. It proves an asymptotic lower bound on the spectral gap. Namely, it shows that
\begin{equation}\label{0.5}
\lim\limits_{\sigma \r 0} \; \limsup\limits_{\ell \r \infty} \, \IP[\lambda_{1,\o} (B_\ell) < \infty \;\mbox{and} \ \lambda_{2,\o} (B_\ell) - \lambda_{1,\o} (B_\ell) < \sigma(\log \ell)^{-(1 + 2/d)} ] = 0,
\end{equation}
and as an immediate consequence that for any $a_\ell = o((\log \ell)^{-(1+ 2/d)})$, as $\ell \r \infty$,
\begin{equation}\label{0.6}
\lim\limits_{\ell \r \infty} \; \IP[\lambda_{1,\o} (B_\ell) < \infty \;\mbox{and} \ \lambda_{2,\o} (B_\ell) - \lambda_{1,\o} (B_\ell)  < a_\ell ] = 0.
\end{equation}

\n
As an aside, it is plausible that  $(\log \ell)^{-(1 + 2/d)}$ captures the correct size of the spectral gap, so that its product with $(\log \ell)^{1 + 2/d}$ remains tight as $\ell \r \infty$: the bounds on the fluctuations of $\lambda_{1,\o}(B_\ell)$ derived in Section 3 of \cite{Szni97b} and their link with the behavior of the spectral gap studied here, make a compelling case, see Remark \ref{rem6.5} 1). It should also be pointed out that the discrepancy between the upper and lower bounds in (\ref{0.4}) is much bigger than $(\log \ell)^{-(1 + 2/d)}$, and the proof of (\ref{0.5}) does {\it not} consist in getting a ``good lower bound'' on $\lambda_{2,\o} (B_\ell)$, and a ``good upper bound'' on $\lambda_{1,\o}(B_\ell)$. Perhaps, to illustrate the significance of the scale $(\log \ell)^{-(1 + 2/d)}$, one can mention that when $\lambda (\ell) \sim c_0 \lol^{-2/d}$, as $\ell \r \infty$, and $\lambda '(\ell) = \lambda(\ell) + \sigma \lol^{-(1+2/d)}$, the balls with principal Dirichlet eigenvalues for $-\frac{1}{2}\, \Delta$ corresponding to $\lambda$ and $\lambda '$ have volumes equivalent to $\frac{d}{\nu} \;\lol$, as $\ell \r \infty$, but the difference of their volumes $\o_d(\lambda_d / \lambda)^{d/2} - \o_d (\lambda_d / \lambda ')^{d/2}$ tends to the constant $d^2 \, \sigma / (2 c_0 \, \nu)$, see also (\ref{0.9}).

\medskip
With this in mind let us describe the strategy of the proof of Theorem \ref{6.1} (see (\ref{0.5})). It first involves showing in Theorem \ref{theo4.1}, with the help of the {\it method of enlargement of obstacles}, see Chapter 4 of \cite{Szni98a}, and quantitative Faber-Krahn inequalities, see \cite{BrasPhilVeli15} and \cite{FuscMaggPrat09}, that for large $\ell$ on most of the event in (\ref{0.5}) there are distant balls $\wh{B}$ and $\wh{B}'$ in $B_\ell$ with same radius $\wh{R}$ slightly bigger than $R_0 (\log \ell)^{1/d}$, see (\ref{0.2}), such that the principal Dirichlet eigenvalues $\lambda_{1,\o}(\wh{B})$ and $\lambda_{1,\o}(\wh{B}')$, see (\ref{1.9}), are both within close range (slightly bigger than $\sigma(\log \ell)^{-(1 + 2/d)}$) of $\lambda_{1,\o}(B_\ell)$.

\medskip
Then, from this fact one reduces in Proposition \ref{prop5.3} the task of showing (\ref{0.5}) to the proof of a deconcentration type estimate for the distribution of $\lambda_{1,\o}(D_0)$, with $D_0$ an open box of side-length $L_0 = 10( \lceil R_0 \rceil + 1) (\log \ell)^{1/d}$, see (\ref{0.2}), in a specific deviation regime of low values, with additional information on the corresponding principal Dirichlet eigenfunction. Namely, one introduces a suitable level $t_\ell$ such that $\IP[\lambda_{1,\o} (D_0) \le t_\ell]$ is of order $(\log \ell) / \ell^d$, see (\ref{5.9}) ($t_\ell$ depends on an additional parameter $\Gamma$, which eventually tends to infinity, and $t_\ell$ is equivalent to $c_0(\log \ell)^{-2/d}$, as $\ell \r \infty$). To prove (\ref{0.5}), it then suffices to show the suitable smallness of (with $\wh{\eta}$ a positive real provided by Theorem \ref{theo4.1}):
\begin{equation}\label{0.7}
\limsup\limits_{\ell \rightarrow \infty}  \; \frac{\ell^d}{(\log \ell)} \; \sup\limits_{t \le t_\ell} \; \IP [ \lambda_{1,\o} (D_0) \in [t, t + \ve_\ell], \; \varphi_{1,D_0,\o} \le e^{-(\log \ell)^{\wh{\eta}}} \ \mbox{in} \; D_0 \backslash D_0^{\rm int}] ,
\end{equation}

\n
where $D_0^{\rm int}$ stands for the closed concentric sub-box of $D_0$ of side-length $(2\lceil R_0 \rceil + 4) (\log \ell)^{1/d}$, $\varphi_{1,D_0,\o}$ for the principal Dirichlet eigenfunction attached to $D_0$ in the configuration $\o$, see (\ref{1.14}), and $\ve_\ell$ for a quantity slightly bigger than $\sigma(\log \ell)^{-(1 + 2/d)}$, see (\ref{5.17}).

\medskip
The control of (\ref{0.7}) and the resulting lower bound on the spectral gap (\ref{0.5}), i.e.~Theorem \ref{theo6.1}, hinges on the main deconcentration estimate (\ref{6.2}) in Theorem \ref{theo6.2}, which embodies a central new aspect of this work. It allows to dominate up to a multiplicative constant the probability that appears in (\ref{0.7}) by $\IP[\lambda_{1,\o} (D_0) \in J_i]$ for any $i$, for a suitable (large) collection $J_1,\dots,J_m$ of pairwise disjoint sub-intervals of $(0,t_\ell)$. Whereas increasing $\lambda_{1,\o} (D_0)$ is a comparatively easier task (it can be achieved by the addition of one single obstacle in a location where $\varphi_{1,D_0,\o}$ is not too small), decreasing $\lambda_{1,\o}(D_0)$ in a controlled fashion, as required by the constraint that the many disjoint $J_1,\dots,J_m$ remain in $(0,t_\ell)$, is more delicate. For this task we do not leverage the geometric information concerning the underlying near spherical clearings of Theorem \ref{theo4.1}. This information is anyway too coarse. We instead perform a ``gentle expansion'' of the Poisson cloud by homotheties of ratio $\exp\{u_i / |D_0|\}$ for suitable $u_1 < \dots < u_m$ in $(0,1)$ in the proof of Theorem \ref{theo6.2}, see (\ref{6.49}).  They dilute the Poisson point process and tendentially decrease eigenvalues. This is tailored so that the expansion corresponding to $u_i$ takes $\lambda_{1,\o}(D_0)$ from $[t, t + \ve_\ell]$ to $J_i$.
Incidentally, the constraint on $\varphi_{1,D_0,\o}$ in (\ref{0.7}) is important and ensures a proper centering of the underlying clearing in $D_0$, shielding it from damage at the boundary of $D_0$ under the gentle expansion. We also refer to Lemma 3.3 of \cite{DingFukuSunXu21} for other kinds of transformations involving the removal of obstacles, which however do not seem adequate for the task of proving (\ref{6.2}), and to the proof of Proposition 1.12 of \cite{DumiRiveRodrVann21} for a recent instance of deconcentration estimates in a percolation context. Let us also mention that unlike the results of the previous sections, which rather straightforwardly could be adapted to the case of Poissonian soft obstacles, see (\ref{1.18}), the proof of the deconcentration estimates in Theorem \ref{theo6.2} makes a genuine use of the hard spherical Poissonian obstacles (and the continuity of the space, as in \cite{DumiRiveRodrVann21}).

\medskip
Owing to the work Kerner-Pechmann-Spitzer \cite{KernPechSpit20}, our results have a natural application to a version of the problem investigated by Kac and Luttinger in \cite{KacLutt73},  \cite{KacLutt74} concerning the Bose-Einstein condensation of a gas of non-interacting bosons among hard obstacles corresponding to balls of radius $a > 0$ centered at the points of a Poisson cloud of intensity $\nu > 0$ on $\IR^d$, $d \ge 2$. In the one-dimensional case we refer to the results in \cite{KernPechSpit19b}, \cite{KernPechSpit20} for the related Luttinger-Sy model, and in \cite{Pech20} for soft Poissonian obstacles.

\medskip
In the model under consideration here, one knows, see (\ref{7.9}), (\ref{7.10}), that the density of states $m(d \lambda)$ is a measure on $\IR_+$ with Laplace transform: 
\begin{equation}\label{0.8}
\dis\int_{[0,\infty)} e^{-t \lambda}\; m(d\lambda) = (2 \pi t)^{-d/2} \,E^t_{0,0} [\exp \{ - \nu \, |W_t^a|\}], \; \mbox{for $t > 0$},
\end{equation}

\n
where $E^t_{0,0}$ denotes the expectation for a Brownian bridge in time $t$ in $\IR^d$ from the origin to itself, and $W_t^a$ stands for the Wiener sausage in time $t$ and radius $a$ of that bridge (i.e. the closed $a$-neighborhood of the bridge trajectory). The density of states has a so-called {\it Lifshitz tail behavior}, see for instance Corollary 3.5 of \cite{Szni90a}, the original proof going back to the work of Donsker-Varadhan \cite{DonsVara75b}, see also Chapter 10.B of \cite{PastFigo92}:
\begin{equation}\label{0.9}
m([0,\lambda]) = \exp\big\{- \nu \o_d(\lambda_d/\lambda)^{d/2} \big(1 + o(1)\big)\big\}, \; \mbox{as $\lambda \r 0$}.
\end{equation}
One then has (see Section 2 of  \cite{KernPechSpit20}) a finite critical density for the system at inverse temperature $\beta > 0$:
\begin{equation}\label{0.10}
\rho_c (\beta) = \dis\int^\infty_0 (e^{\beta \lambda} - 1)^{-1} \,m(d \lambda) < \infty.
\end{equation}
Thus, in the thermodynamic limit, for a fixed density $\rho \in (0,\infty)$, one lets the particle number $N \ge 1$ and the length scales $\ell_N$, such that $N = \rho \,|B_{\ell_N}|$ tend to infinity. One then defines suitably truncated $\ov{\lambda}_{j,\o} (B_{\ell_N})$, $j \ge 1$, see (\ref{7.5}), which coincide with the $\lambda_{j,\o}(B_{\ell_N})$, $j \ge 1$, when $B_{\ell_N,\o} \not= \emptyset$ (an increasing sequence of events with $\IP$-probability tending to $1$), and corresponding occupation numbers $\ov{n}_N^{\,j,\o}$ of the $j$-th eigenstate for a grand-canonical version of non-interacting bosons with density $\rho$ in $B_{\ell_N}$ in the presence of the spherical impurities attached to $\o$, see (\ref{7.7}). As an application of the lower bound on the spectral gap in Theorem \ref{theo6.1} and the results in \cite{KernPechSpit20}, we show in Theorem \ref{theo7.1} that when $\rho > \rho_c(\beta)$ a so-called type-I Bose-Einstein condensation in probability in a single mode takes place:
\begin{align}
\mbox{as $N \r \infty$}, &\; \mbox{$\ov{n}\,^{1,\o}_N / N$ tends to $\big(\rho - \rho_c(\beta)\big)/ \rho$ in $\IP$-probability, and} \label{0.11}
\\
&\; \mbox{for all $j \ge 2$, $\ov{n}^{\,j,\o} / N$ tends to $0$ in $\IP$-probability.} \nonumber
\end{align}
\n
This seems to be the first multi-dimensional example in the natural class of random Poissonian obstacles for which type-I condensation has been established, see \cite{KernPechSpit20}.

\medskip

We will now describe the organization of this article. Section 1 collects notation as well as some basic results concerning the Dirichlet eigenvalues and the eigenfunctions under consideration. It also states the quantitative Faber-Krahn inequality of \cite{BrasPhilVeli15}, see (\ref{1.17}). Section 2 briefly recalls the results of the method of enlargement of obstacles from Chapter 4 of \cite{Szni98a} that will be used in Sections 3 and 4. Section 3 introduces a certain event $T$ in (\ref{3.24}), encapsulating typical configurations in scale $\ell$, of probability tending to $1$ as $\ell$ goes to infinity. In Section 4, the main result is Theorem \ref{theo4.1}, which in particular reduces the analysis to the consideration of two distant balls of radius $\wh{R}$ slightly bigger than $R_0 (\log \ell)^{1/d}$ within $B_\ell$, with principal Dirichlet eigenvalues within close range (almost $\sigma (\log \ell)^{-(1+2/d)}$) of $\lambda_{1,\o} (B_\ell)$. In Section 5, the proof of the lower bound on the spectral gap is reduced to proving (\ref{0.7}) in Proposition \ref{prop5.3}. Section 6 proves the deconcentration estimates in Theorem \ref{theo6.2} from which the lower bound on the spectral gap, see (\ref{0.5}), follows in Theorem \ref{theo6.1}. Section 7 contains the application to the Bose-Einstein condensation for the version of the Kac-Luttinger model, which we consider.

\medskip
Finally, throughout the article we denote by $c, \wt{c}, c', \dots$ positive constants changing from place to place, which depend on $d$. From Section 3 onwards they will also implicitly depend, unless otherwise stated, on the parameters selected there in the context of the method of enlargement of obstacles. As for numbered constants such as $c_0,c_1,c_2,\dots$, they refer to the value corresponding to their first appearance in the text (for instance see (\ref{0.3}) for $c_0$).

\bigskip\n
{\bf Acknowledgements:} The author wishes to thank Joachim Kerner, Maximilian Pechmann, and Wolfgang Spitzer for several stimulating discussions during the conference ``On mathematical aspects of interacting systems in low dimension'' that took place in Hagen in June 2019.

\section{Set-up and some useful facts}
\setcounter{equation}{0}

In this section we collect further notation as well as some results concerning eigenvalues and eigenfunctions. We also state the quantitative Faber-Krahn Inequality at the end of the section. Throughout we assume that $d \ge 2$.

\medskip
When $I$ is a finite set we let $| I |$ stand for the number of elements of $I$. We write $\IN = \{0,1,\dots \}$ for the set of non-negative integers. For $a,b$ real numbers we write $a \wedge b$, resp. $a \vee b$, for the minimum, resp. the maximum, of $a$ and $b$. Given $r \ge 0$ we denote by $[r]$ the integer part of $r$ and by $\lceil r \rceil$ the ceiling of $r$. For $(a_n)_{n \ge 1}$ and $(b_n)_{n\ge 1}$ positive sequences, the notation $a_n \gg b_n$ or $b_n = o(a_n)$ means that $\lim_n b_n/a_n = 0$. We write $| \cdot |$ and $| \cdot |_\infty$ for the Euclidean and the supremum norms on $\IR^d$. We denote by $B(x,r)$ and $ \overset{\circ}{B} (x,r)$ the closed and open Euclidean balls with center $x \in \IR^d$ and radius $r \ge 0$. We write $B_\infty (x,r)$ and $ \overset{\circ}{B}_\infty (x,r)$ in the case of the supremum norm, and also refer to them as the closed and open boxes with center $x$ and side-length $2r$. Given $A, B \subseteq \IR^d$, we denote by $d(A,B) = \inf\{ |x-y|; x \in A,y \in B\}$ the mutual Euclidean distance between $A$ and $B$, and by ${\rm diam}(A) = \sup \{|x-y|; x,y \in A\}$ the diameter of $A$. We define $d_\infty (A,B)$ and ${\rm diam}_\infty (A)$ in an analogous fashion with $|\cdot |_\infty$ in place of $| \cdot |$. We write $\cB(\IR^d)$ for the collection of Borel subsets of $\IR^d$, and for $A \in B(\IR^d)$ we let $|A|$ stand for the Lebesgue measure of $A$ (hopefully this causes no confusion with the notation for the cardinality of $A$). When $f$ is a function $f_+ = \max \{f,0\}, f_- = \max \{-f,0\}$ stand for the positive and negative parts of $f$, and $\|f\|_\infty$ for the supremum norm of $f$. When $1 \le p < \infty$ and $A \in B(\IR^d)$ we denote by $L^p(A)$ the $L^p$-space of $p$-integrable functions for the Lebesgue measure that vanish outside of $A$, and write $\| \cdot \|_p$ for the $L^p$- norm. Given $U$ an open subset of $\IR^d$, we write $H^1(U)$ and $H^1_0(U)$ for the usual Sobolev spaces (corresponding to $W^{1,2} (U)$ and $W_0^{1,2}(U)$ in \cite{Adam75}, p.~45).

\medskip
We turn to the description of the random medium. The canonical space $\Omega$ consists of locally finite, simple point measures on $\IR^d$, endowed with the canonical $\sigma$-algebra $\cG$ generated by the applications $\o \in \Omega \r \o(A) \in \IN \cup \{\infty\}$, for $A = \cB (\IR^d)$. We routinely write $\o = \Sigma_i \,\delta_{x_i}$ for a generic $\o \in \Omega$, and $x \in \o$ to denote that $x$ belongs to ${\rm supp} \,\o$ (the support of $\o$). On $\Omega$ endowed with the above $\sigma$-algebra, we let (see also \cite{LastPenr18}):
\begin{equation}\label{1.1}
\begin{array}{l}
\mbox{$\IP$ stand for the law of the Poisson point process on $\IR^d$ with constant}
\\
\mbox{intensity $\nu > 0$.} 
\end{array}
\end{equation}

\n
The radius of the obstacles is given by
\begin{equation}\label{1.2}
a > 0,
\end{equation}

\n
and the {\it obstacle set} in the configuration $\o$ is the closed subset of $\IR^d$
\begin{equation}\label{1.3}
Obs_\o = \bigcup_{x \in \o} B(x,a).
\end{equation}
For $U$ an open subset of $\IR^d$ and $\o \in \Omega$, we write
\begin{equation}\label{1.4}
U_\o = U \;\backslash \;Obs_\o 
\end{equation}

\n
for the possibly empty open subset remaining after deletion of the obstacle set. When $U$ is bounded, $U_\o$ has finitely many connected components (since $\o$ is locally finite). Of special interest for us is the case when $U$ is an open box of side-length $2 \ell$ centered at the origin:
\begin{equation}\label{1.5}
B_\ell = (-\ell, \ell)^d, \; \ell \ge 10.
\end{equation}
Note that $B_{\ell,\o}$ is non-decreasing in $\ell$, and that by a routine Borel-Cantelli type argument $\IP$-a.s., $B_{\ell,\o}$ is not empty for large $\ell$, i.e.
\begin{equation}\label{1.6}
\mbox{$\IP(\Omega_\infty) = 1$ where $\Omega_\infty = \{\o \in \Omega; B_{\ell,\o} \not= \emptyset$, for large $\ell\}$}.
\end{equation}

\n
We then proceed with some notation concerning Brownian motion. We denote by
\begin{equation}\label{1.7}
p(t,x,y) = (2 \pi t)^{-d/2} \exp\big\{ - \mbox{\f $\dis\frac{|y - x|^2}{2t} $}\big\}, \; t > 0, \, x,y \in \IR^d,
\end{equation}

\n
the transition density for Brownian motion. When $x \in \IR^d$, we let $P_x$ stand for the Wiener measure starting from $x$, i.e.~the canonical law of Brownian motion starting at $x$ on the space $W = C(\IR_+, \IR^d)$ of continuous $\IR^d$-valued trajectories. We write $(X_t)_{t \ge 0}$ for the canonical process, $(\cF_t)_{t \ge 0}$ for the canonical right-continuous filtration, and $(\theta_t)_{t \ge 0}$ for the canonical shift. Given an open subset $U$ of $\IR^d$ and $w \in W$, we denote by $T_U(w) = \inf\{s \ge 0; X_s(w) \notin U\}$ the exit time from $U$. When $F$ is a closed subset of $\IR^d$ and $w \in W$, we write $H_F(w) = \inf\{s \ge 0; X_s(w) \in F\}$ for the entrance time in $F$. These are $(\cF_t)$-stopping times.

\medskip
For $U$ open subset of $\IR^d$ and $\o \in \Omega$, we denote by
\begin{equation}\label{1.8}
r_{U,\o} (t,x,y) \; \big(\le p(t,x,y)\big) \; \mbox{for} \; t > 0, \, x,y \in \IR^d,
\end{equation}

\n
the transition kernel of Brownian motion killed outside $U_\o$, see (3.4), p.~13 of \cite{Szni98a}. It is jointly measurable, symmetric in $x,y$, it vanishes if $x$ or $y$ does not belong to $U_\o$, satisfies the Chapman-Kolmogorov equations, see pp.~13,14 of \cite{Szni98a}, and it is a continuous function of $(t,x,y)$ in $(0,\infty) \times U_\o \times U_\o$, see Proposition 3.5, p.~18 of \cite{Szni98a}. If $U$ satisfies an exterior cone condition, so does $U_\o$, and the proof of Proposition 3.5, p.~18 of \cite{Szni98a} can be adapted (see (3.20) on p.~18 of this reference) to show that $r_{U,\o} (t,x,y)$ is a continuous function on $(0,\infty) \times \IR^d \times \IR^d$. We will use this fact when $U$ is an open box in $\IR^d$.

\medskip
We now proceed with the discussion of the eigenvalues and eigenfunctions. Given a bounded open subset $U$ of $\IR^d$, we denote by
\begin{equation}\label{1.9}
0 < \lambda_{1,\o} (U) \le \lambda_{2,\o}(U) \le \dots \le \lambda_{i,\o}(U) \le \dots \le \infty,
\end{equation}

\n
the successive Dirichlet eigenvalues of $-\frac{1}{2} \,\Delta$ in $U_\o$, repeated according to their multiplicity, with the convention that $\lambda_{i,\o} (U) = \infty$ for all $i \ge 1$, if $U_\o = \emptyset$. They are measurable in $\o$ (as follows for instance from the min-max principles, see Version 3 in Theorem 12.1, p.~301 of \cite{LiebLoss01}). Note that the open set $U_\o$ need not be connected (even when $U$ is connected) and the $\lambda_{i,\o}(U), i \ge 1$, correspond to the non-decreasing reordering of the collection of Dirichlet eigenvalues of $-\frac{1}{2} \, \Delta$ in the finitely many connected components of $U_\o$. Also some of the $\lambda_{i,\o}(U)$, $i \ge 1$, may not be simple, for instance in the case $U = B_\ell$, and no obstacle falls into $B_\ell$, i.e. when $B_{\ell,\o} = B_\ell$, an event of positive probability.

\medskip
When $U_\o \not= \emptyset$ and $\lambda$ is a Dirichlet eigenvalue of $-\frac{1}{2} \, \Delta$ in $U_\o$, an eigenfunction $\varphi$ attached to $\lambda$ will be square integrable and satisfy
\begin{equation}\label{1.10}
\varphi(x) = e^{\lambda t} \dis\int r_{U,\o} (t,x,y) \,\varphi(y)\, dy, \; \mbox{for all $t > 0$ and $x \in \IR^d$}.
\end{equation}

\n
In particular, we will implicitly choose the version of $\varphi$, which is continuous in $U_\o$ and identically equal to $0$ outside $U_\omega$. When $U$ satisfies an exterior cone condition and $U_\omega \not= \emptyset$, such an eigenfunction $\varphi$ will be continuous on $\IR^d$ by the remark above (\ref{1.9}). The next lemma will be used repeatedly.

\begin{lemma}\label{lem1.1}
When $U$ is a bounded open subset of $\IR^d$, $\o \in \Omega$, and $U_\o \not= \emptyset$, then for any Dirichlet eigenvalue $\lambda $ of $-\frac{1}{2} \,\Delta$ in $U_\o$, and eigenfunction $\varphi$ for $\lambda$ of unit $L^2$-norm, one has
\begin{equation}\label{1.11}
\|\varphi\|_\infty \le c_1 \, \lambda^{d/4} \; \mbox{(with $c_1 = (4 \pi)^{-d/4} \,e)$}.
\end{equation}

\n
Moreover, when $O$ is a connected component of $U_\o$, and $O'$ an open subset of $O$ with principal Dirichlet eigenvalue $\lambda '$ for $-\frac{1}{2} \, \Delta$ bigger than $\lambda$, then
\begin{equation}\label{1.12}
\varphi(y) = E_y [\varphi(X_{T_{O'}}) \; \exp \{\lambda T_{O'}\}, \, T_{O'} < T_{U_\o}],\;\mbox{for all $y \in O'$}.
\end{equation}
\end{lemma}

\begin{proof}
We first prove (\ref{1.11}). We use (\ref{1.10}) with $t = \lambda^{-1}$, so that for any $x \in U_\o$ one has:
\begin{equation}\label{1.13}
\begin{split}
\varphi(x)  = &\; e \dis\int r_{U,\o} (t,x,y) \,\varphi(y) \, dy \stackrel{\rm Cauchy-Schwarz}{\le} 
\\
&\; e \;\big(\dis\int r_{U,\o}  (t,x,y)^2 dy\big)^{1/2} \big(\dis\int \varphi^2(y) \,dy\big)^{1/2} 
\\
&\hspace{-4.5ex} \stackrel{\rm symmetry}{=} e\; \Big(\dis\int r_{U,\o} (t,x,y) \,r_{U,\o} (t,y,x)\,dy \Big)^{1/2}  \stackrel{\rm Chapman-Kolmogorov}{=} e \,r_{U,\o} (2t,x,x)^{1/2}
\\
&\!\!\!\! \!\le \; e \,p(2t,x,x)^{1/2} = e(4 \pi t)^{-d/4} = e (\lambda / 4 \pi)^{d/4},
\end{split}
\end{equation}
and (\ref{1.11}) follows.

\medskip
As for the identity (\ref{1.12}), it follows from the application of (1.54), p.~107 of \cite{Szni98a}.
\end{proof}

\medskip
As already mentioned, when $U$ is a bounded open set and $\o \in \Omega$, the open set $U_\o$ need not be connected, and the eigenvalue $\lambda_{1,\o}(U)$ need not be simple. To take care of this feature, when $U_\o \not= \emptyset$, we denote by $\varphi_{1,U,\o}$ the $L^2$-normalized orthogonal projection of the function $1$ on the eigenspace attached to $\lambda_{1,\o}(U)$:
\begin{equation}\label{1.14}
\varphi_{1,U,\o}(x) = \left\{ \begin{array}{l}
\lim_{t \r \infty} \dis\int r_{U,\o} (t,x,y)\,dy\; \Big/\; \Big\|\dis\int r_{U,\o} (t, \cdot , y) \,dy \Big\|_2 \; \mbox{if $x \in U_\o$},
\\[2ex]
\mbox{$0$, if $x \notin U_\o$}. 
\end{array}\right.
\end{equation}

\n
Note that $\varphi_{1,U,\o}(x)$ is positive exactly when $x$ belongs to a connected component of $U_\o$ with principal Dirichlet eigenvalue for $-\frac{1}{2}\,\Delta$ equal to $\lambda_{1,\o}(U)$. By convention, when $U_\o = \emptyset$, we simply set $\varphi_{1,U,\o} = 0$.

\medskip
From time to time for $U$ bounded open subset of $\IR^d$, we will use the notation 
\begin{equation}\label{1.15}
\lambda_{-\frac{1}{2} \,\Delta} (= \lambda_{1,\o = 0}(U)) \;\mbox{for the principal Dirichlet eigenvalue of $-\frac{1}{2} \, \Delta$ in $U$},
\end{equation}
so that in (\ref{0.3}) with the notation from the beginning of this section
\begin{equation}\label{1.16}
\lambda_d = \lambda_{-\frac{1}{2}\,\Delta} \big(\overset{\circ}{B} (0,1)\big).
\end{equation}

\n
We proceed with the statement of the quantitative Faber-Krahn inequality. The classical Faber-Krahn inequality states that for $U$ bounded open subset of $\IR^d$, $\lambda_{-\frac{1}{2}\, \Delta} (U)$ is bigger or equal to the principal Dirichlet eigenvalue of $-\frac{1}{2} \,\Delta$ in an open ball of same volume as $U$, that is $\lambda_d \,(\o_d / |U|)^{2/d}$. We will use in the proofs of Theorems \ref{theo4.1} and \ref{theo4.2} the following {\it quantitative version of Faber-Krahn's inequality}, see the Main Theorem on p.~1781 in \cite{BrasPhilVeli15}: One has a dimension dependent constant $c_2$ such that for any bounded non-empty open subset $U$ of $\IR^d$:
\begin{equation}\label{1.17}
\lambda_{-\frac{1}{2} \,\Delta} (U) \big(\mbox{\f $\dis\frac{|U|}{\o_d}$}\big)^{2/d} \; / \; \lambda_d - 1 \ge c_2 \,A(U)^2,
\end{equation}

\n
where $A(U) = \inf \{\frac{|U \Delta B|}{|B|}$; $B$ a ball with $|B| = |U|\}$ is the {\it Fraenkel asymmetry} of $U$ (and $\Delta$ stands for the symmetric difference).

\medskip
The Theorem 1.1 of \cite{FuscMaggPrat09} states a similar inequality with $A(U)^2$ replaced by $A(U)^4$ (and a different constant), which would also suffice for our purpose in Section 4.

\medskip
Finally, in several places we refer to {\it soft obstacles}. This corresponds to the case when we have a function $W(\cdot)$, non-negative, bounded, measurable, compactly supported, and not a.e.~equal to zero, and for each $\o = \sum_i \delta_{x_i}$ in $\Omega$ we consider the non-negative locally bounded function (the {\it random potential}\,):
\begin{equation}\label{1.18}
V(x,\o) = \sum\limits_i \,W(x - x_i), \, x \in \IR^d.
\end{equation}
The objects of study are now the Dirichlet eigenvalues and corresponding eigenfunctions of $- \frac{1}{2}\, \Delta + V (\cdot, \o)$ in bounded open subsets $U$ of $\IR^d$. The {\it hard obstacles} under consideration in this article informally correspond to the choice $W(\cdot) = \infty \,1_{\{|\cdot | \le a\}}$.

\section{Inputs from the method of enlargement of obstacles}
\setcounter{equation}{0}

In this section we collect various facts from the method of enlargement of obstacles, see Chapter 4, Sections 1 to 3 of \cite{Szni98a}, which will be employed or underpin some of the results in the next two sections.

\medskip
We begin with an informal description. In a nutshell the method of enlargement of obstacles is a procedure, which for $\ell$ positive real (say bigger than $10$) and a configuration $\o \in \Omega$ attaches in a measurable fashion two disjoint subsets of $\IR^d$ the {\it density set} $\cD_\ell(\o)$ and the {\it bad set} $\cB_\ell(\o)$, so that $\o$ has no point outside $\cD_\ell(\o) \cup \cB_\ell(\o)$:
\begin{equation}\label{2.1}
\o \big(\IR^d\, \backslash \, (\cD_\ell(\o) \cup \cB_\ell(\o))\big) = 0 \; \mbox{and} \; \cD_\ell(\o) \cap \cB_\ell (\o) = \emptyset .
\end{equation}

\n
In the presentation made here $(\log \ell)^{1/d}$ corresponds to the unit scale in Chapter 4 of \cite{Szni98a} and $\ve = (\log \ell)^{-1/d}$ to the small parameter in the same reference. The statements recalled below will hold in the large $\ell$ limit (i.e.~small $\ve$ limit) but uniformly in $\o$.

\medskip
In essence, adding Dirichlet boundary conditions on $\overline{\cD_\ell(\o)}$ (the closure of $\cD_\ell(\o)$) does not increase too much Dirichlet eigenvalues of type $\lambda_{1,\o}(U)$ when they are below $M(\log \ell)^{-2/d}$ (see Theorem \ref{theo2.1} below), and the bad set $\cB_\ell(\o)$ has small relative volume on each box of side-length $(\log \ell)^{1/d}$ in $\IR^d$, see Theorem \ref{theo2.4}. In addition, the sets $\cD_\ell(\o)$ and $\cB_\ell(\o)$ have ``low combinatorial complexity'': their restriction to each cube
\begin{equation}\label{2.2}
C_q = (\log \ell)^{1/d} \,(q + [0,1)^d), \; q \in \IZ^d,
\end{equation}

\n
is a union of disjoint cubes in an $L$-adic decomposition of $C_q$, of size larger than (and of order) $(\log \ell)^{(1 - \gamma)/d}$ in the case of the density set $\cD_\ell(\o)$, and $(\log \ell)^{(1 - \beta)/d}$ in the case of the bad set $\cB_\ell(\o)$, where $0 < \gamma < \beta < 1$. This feature constrains the number of possible shapes of the restriction of $\cD_\ell(\o)$ and of $\cB_\ell(\o)$ to any such cube $C_q$, to at most $2^{(\log \ell)^\gamma}$ in the case of $\cD_\ell(\o)$, and at most $2^{(\log \ell)^\beta}$ in the case of $\cB_\ell(\o)$, see (\ref{2.12}). This reduced combinatorial complexity of the density set and of the bad set underpins the {\it coarse graining} aspect of the method, and its power when bounding the probability of events of a large deviation nature.

\medskip
We now turn to the precise statements that will be helpful for us in the next two sections. One first selects parameters that fulfill the requirements in (3.66), p.~\!181 of \cite{Szni98a} (see also (3.27), p.~\!173 and (3.64), p.~\!180). These parameters are $0 < \alpha < \gamma < \beta < 1$, an integer $L \ge 2$ (entering the $L$-adic decomposition of the boxes $C_q$, $q \in \IZ^d$ in (\ref{2.2})), $\delta > 0$ (entering the definition of the density set), $\rho > 0$ (governing the quality of the eigenvalue estimates), $\kappa > 0$ (governing the local volume of the bad set). As mentioned above they are chosen so as to satisfy (3.66), p.~181 of \cite{Szni98a}. With this choice performed, the results in Chapter 4, Sections 2 and 3 of \cite{Szni98a} apply in the context of the hard spherical obstacles considered here (see (\ref{1.3})). They yield the following statements:

\begin{theorem}\label{theo2.1} {\it (Eigenvalue estimate)}

\medskip\n
For any $M > 0$,
\begin{equation}\label{2.3}
\lim\limits_{\ell \r \infty} \; \sup\limits_{\o, U}\, (\log \ell)^{\rho/d} \, \big[\big\{ \lambda_{1,\o} \big(U \,\backslash \, \overline{\cD_\ell(\o)}\big) (\log \ell)^{2/d} \big\} \wedge M - \{\lambda_{1,\o} (U) (\log \ell)^{2/d} \big\} \wedge M\big] = 0,
\end{equation}
where in the supremum $\o$ runs over $\Omega$ and $U$ over all bounded open sets in $\IR^d$ (and  $\wedge$ refers to the minimum, see the beginning of Section 1).
\end{theorem}

\medskip
See Theorem 2.3, p.~158 of \cite{Szni98a} for the proof. In the next sections we will only need the value $M = 2 c_0$ (with $c_0$ from (0.3)). The statement above can actually be extended to arbitrary open sets $U$, but here in (\ref{1.9}) we have only defined $\lambda_{1,\o}(U)$ for bounded open sets $U$, a general enough set-up for our purpose.

\medskip
One also has an estimate, which provides a lower bound on the probability that Brownian motion enters the obstacle set before moving at distance $L(\log \ell)^{(1-\alpha)/d}$:

\begin{lemma}\label{lem2.2} For large $\ell$ one has that for any $\o \in \Omega$ and $x \in \overline{\cD_\ell(\o)}$,
\begin{equation}\label{2.4}
\begin{array}{l}
P_x [ \tau_{L(\log \ell)^{(1-\alpha)/d}} < H_{Obs_{\o}}] \le \frac{1}{2}, \; \mbox{where}
\\[1ex]
\mbox{for $u > 0$, $\tau_u = \inf\{s \ge 0; \,|X_s - X_0|_\infty \ge u \}$ and we recall the notation (\ref{1.3}), and}
\\[0.5ex]
\mbox{$H_{Obs_{\o}}$ is the entrance time in $Obs_\o$, see above (\ref{1.8}).}
\end{array}
\end{equation}
\end{lemma}

For the proof, see Lemma 2.1, p.~154 (and (2.19)', p.~157) of \cite{Szni98a}. One actually has a much stronger estimate in the quoted reference, but (\ref{2.4}) will suffice for our purpose.

\medskip
The next result that we quote corresponds to the case of
\begin{equation}\label{2.5}
\mbox{$U_2 \supseteq U_1$ bounded open sets and $\cA$ a closed set in $\IR^d$},
\end{equation}
so that outside of $\cA \cup\overline{\cD_\ell(\o)}$, $U_2$ has ``small relative volume'' in all boxes $C_q$, $q \in \IZ^d$, namely, one has $r > 0$ with
\begin{equation}\label{2.6}
\sup\limits_{q \in \IZ^d} \,| (U_2 \, \backslash \, (\cA \cup\overline{\cD_\ell(\o)} )| \cap C_q | < r^d,
\end{equation}
and in addition, one has $R > 0$ so  that $U_1$ contains the trace on $U_2$ of an $R$-neighborhood of $\cA \cap U_2$ for the supremum distance, that is 
\begin{equation}\label{2.7}
d_\infty (U_2 \, \backslash \, U_1, \, \cA \cap U_2) \ge R.
\end{equation}
The next theorem provides a setting in which $\lambda_{1,\o}(U_1)$ is not much bigger than $\lambda_{1,\o}(U_2)$. Once again, only the choice $M = 2 c_0$ will be used in Sections 3 and 4.

\begin{theorem}\label{theo2.3}
For any $M > 0$ there exists constants $c_3(d) > 0$, $c_4(d,M) \in (1,\infty)$, $r_0(d,M) \in (0, \frac{1}{4})$ such that 
\begin{equation}\label{2.8}
\lim\limits_{\ell \r \infty} \; \wt{\sup} \; \exp\big\{ c_3 \big[ \mbox{\f $\dis\frac{R}{4r}$}\big] \big\} [\{\lambda_{1,\o} (U_1) (\log \ell)^{2/d}\} \wedge M - \{\lambda_{1,\o} (U_2) (\log \ell)^{2/d} \} \wedge M] \le 1,
\end{equation}
where $\wt{\sup}$ denotes the supremum over all $\o \in \Omega$, $U_2 \supseteq U_1$, $\cA$, $R > 0$, $r > 0$, such that (\ref{2.5})~-~(\ref{2.7}) hold and
\begin{align}
& L(\log \ell)^{(1-\alpha)/d} < r < r_0 (\log \ell)^{1/d} \; \mbox{(recall $L$ governs the $L$-adic decomposition)}, \label{2.9}
\\[1ex]
&\mbox{\f $\dis \frac{R}{4r} $} > c_4. \label{2.10}
\end{align}
\end{theorem}

For the proof we refer to Theorem 2.6, p.~164 of \cite{Szni98a}. Concerning the volume estimate for the bad set, one has

\begin{theorem}\label{theo2.4}
\begin{equation}\label{2.11}
\lim\limits_{\ell \r \infty} \; \sup\limits_{q \in \IZ^d, \o \in \Omega} (\log \ell)^{\kappa/d} \; \mbox{\f $\dis\frac{|\cB_\ell (\o) \cap C_q|}{|C_q|} $}= 0.
\end{equation}
\end{theorem}

For the proof, see Theorem 3.6, p.~181 of \cite{Szni98a}.

\medskip
As for the combinatorial complexity of the density set and the bad set, one has
\begin{equation}\label{2.12}
\begin{array}{l}
\mbox{for each $\ell \ge 10$, $q \in \IZ^d$ and $\o \in \Omega$, the sets $C_q \cap \cD_\ell(\o)$ and $C_q \cap \cB_\ell(\o)$ take}
\\
\mbox{at most $2^{(\log \ell)^\beta}$ possible shapes}.
\end{array}
\end{equation}

\n
These shapes corresponds to the various unions of $L$-adic subboxes of $C_q$ of the same size, which is bigger or equal to $(\log \ell)^{(1 - \gamma)/d}$ in the case of the density set, see (2.7) and (2.13), pp.~\!151-152 of \cite{Szni98a}, and bigger or equal to $(\log \ell)^{(1- \beta)/d}$ in the case of the bad set, see (3.43)~-~(3.46), p.~\!177 of the same reference.

\section{Setting up typical configurations}
\setcounter{equation}{0}

In this section we collect some first consequences of the method of enlargement of obstacles. We will introduce an event depending on $\ell$, of high probability as $\ell \r \infty$, see (\ref{3.24}), (\ref{3.25}), which will encapsulate the nature of ``typical configurations'', and will be very convenient in the analysis of the first and second Dirichlet eigenvalues $\lambda_{1,\o} (B_\ell)$ and $\lambda_{2,\o}(B_\ell)$ for large $\ell$, see (\ref{1.9}), (\ref{1.5}) for notation. With the exception of Lemma \ref{lem3.1}, we mainly collect here results from Section 4 in Chapter 4 of \cite{Szni98a}, which although written in the context of soft obstacles remains valid in the (simpler) context of hard spherical obstacles. We recall that $a > 0$ denotes the radius of these spheres, $\nu > 0$ the intensity of the Poisson point process, and we have selected a fixed choice of admissible parameters $0 < \alpha < \gamma < \beta < 1$, $L \ge 2$ integer, $\delta > 0$, $\rho > 0$, $\kappa > 0$ that satisfy the requirements in (3.66), p.~181 of \cite{Szni98a} so that the results stated in the previous section apply. We further choose as in (4.41), p.~189 of \cite{Szni98a}
\begin{equation}\label{3.1}
\beta' \in (\beta, 1) .
\end{equation}
Unless otherwise specified, the positive constants will implicitly depend on the dimension $d$ and the above parameters, as explained at the end of the Introduction. We recall that throughout we assume $d \ge 2$ and $\ell > 10$. Corresponding to $r$ in (4.22), p.~186 of \cite{Szni98a}, we have a (``small enough'', see the quoted reference)
\begin{equation}\label{3.2}
r_1(d,\nu) \in \big(0, r_0(d, M = 2c_0)\big)
\end{equation}
(with $r_0$ as in Theorem \ref{theo2.3}) and we define $R_1(d,\ell,\nu)$ corresponding to $R$ in (4.23), p.~186 of \cite{Szni98a}, as the smallest positive integer for which, in the notation of Theorem \ref{theo2.3}, 
\begin{equation}\label{3.3}
\mbox{\f $\dis\frac{R_1}{4r_1}$} > c_4 (d, M = 2c_0) \; \mbox{and} \; c_3(d) \; \big[\mbox{\f $\dis\frac{R_1}{4r_1}$}\big] \ge 3 \log \log \ell .
\end{equation}
We then define the random open set
\begin{equation}\label{3.4}
\begin{split}
O = & \; \mbox{the open $R_1(\log \ell)^{1/d}$-neighborhood for the $| \cdot |_\infty$-norm of the union of boxes}
\\
& \; \mbox{$C_q$ (see (\ref{2.2})) for which $|C_q \, \backslash \, \cD_\ell(\o) | \ge r_1^d (\log \ell)$} .
\end{split}
\end{equation}

\n
(In the terminology of \cite{Szni98a} the boxes $C_q$ satisfying the above condition are the so-called {\it clearing boxes}).

\medskip
\begin{samepage}
Then, by (4.25), (4.26) in Proposition 4.2, p.~186 of \cite{Szni98a}, we have a constant 
\begin{equation}\label{3.5}
\gamma_1 (d,\nu) > 0,
\end{equation}
such that setting for $\ell > 10$
\begin{equation}\label{3.6}
\begin{split}
\cC_\ell =& \;\mbox{the collection of boxes $B = (\log \ell)^{1/d} \big(q + (0, [\gamma_1 \log \log \ell])^d\big),\; q \in \IZ^d$}
\\[-0.5ex]
& \;\mbox{that intersect $B_\ell$},
\end{split}
\end{equation}
the event
\begin{equation}\label{3.7}
\begin{split}
C =& \;\mbox{$\{\o \in \Omega$; all connected components of $O$ intersecting $B_\ell$ are contained},
\\[-0.5ex]
&\; \;\,\mbox{in some $B \in \cC_\ell\}$},
\end{split}
\end{equation}
satisfies
\begin{equation}\label{3.8}
\IP[C] \ge 1 - \ell^d, \; \mbox{for large $\ell$}.
\end{equation}
\end{samepage}

\n
Moreover, see Proposition 4.3, p.~188 of \cite{Szni98a}, one can choose a constant $\gamma_2(d,\nu,a) > 0$ such that the event
\begin{equation}\label{3.9}
\begin{array}{l}
F = \\
\Big\{\o \in \Omega; \mbox{\f $\dis\frac{c_0}{(\log \ell)^{2/d}}$} + \mbox{\f $\dis\frac{\gamma_2}{(\log \ell)^{3/d}}$} \ge \inf\limits_{B \in \cC_\ell} \lambda_{1,\o} (B \cap B_\ell) \ge   \lambda_{1,\o}(B_\ell)  \ge \! \inf\limits_{B \in \cC_\ell} \lambda_{1,\o} (B \cap B_\ell) - \delta_\ell\Big\},
\\[2ex]
\;\mbox{with $\delta_\ell = (\log \ell)^{-(2 + 2/d)}$},
\end{array}
\end{equation}
satisfies
\begin{equation}\label{3.10}
\IP [F] \ge 1 - 2 \ell^{-d}, \; \mbox{for large $\ell$}.
\end{equation}

\medskip\n
Further, see (4.42), p.~189 of \cite{Szni98a}, we introduce the event (see Section 2 for notation): 
\begin{equation}\label{3.11}
G = \big\{\o \in \Omega; \sup\limits_{B \in \cC_\ell} \big| B \, \backslash \, \big(\cD_\ell(\o) \cup \cB_\ell(\o)\big)\big| \le \mbox{\f $\dis\frac{d}{\nu}$} \;(\log \ell) + (\log \ell)^{\beta'}\big\}.
\end{equation}
Then, by Lemma 4.4, p.~189 of \cite{Szni98a}, we have
\begin{equation}\label{3.12}
\IP [G] \ge 1 - \exp\big\{- \mbox{\f $\dis\frac{\nu}{2}$} \;(\log \ell)^{\beta'}\big\}, \; \mbox{for large $\ell$}.
\end{equation}

\n
We also wish to discard the boxes in $\cC_\ell$ that are too close to the boundary of $B_\ell$. To this effect, with $\gamma_1, \gamma_2$ as in (\ref{3.5}) and (\ref{3.9}) above, we define the event
\begin{equation}\label{3.13}
\begin{split}
E = & \; \{\mbox{for all $B \in \cC_\ell$ such that $d_\infty (B, B_\ell^c) \le 100 \,[\gamma_1 \log \log \ell ](\log \ell)^{1/d}$, one has}
\\
& \; \; \lambda_{1,\o} (B) > c_0 (\log \ell)^{-2/d} + 100 \gamma_2 (\log \ell)^{-3/d}\},
\end{split}
\end{equation}
and introduce the sub-collections of {\it interior boxes} and {\it boundary boxes} in $\cC_\ell$:
\begin{equation}\label{3.14}
\begin{array}{l}
\cC^{\rm int}_\ell = \{B \in \cC_\ell, d_\infty(B, B^c_\ell) > 100 \,[\gamma_1 \log \log \ell ] (\log \ell)^{1/d}\} \;\mbox{and}
\\[1ex]
\cC_\ell^{\rm bound} = \cC_\ell \, \backslash \, \cC_\ell^{\rm int}.
\end{array}
\end{equation}
One then has

\begin{lemma}\label{lem3.1}
\begin{equation}\label{3.15}
\mbox{For large $\ell$, $\IP[E] \ge 1 - \exp\big\{ - \mbox{\f $\dis\frac{\nu}{2}$} \;(\log \ell)^{\beta'}\big\}$}.
\end{equation}
\end{lemma}

\begin{proof}
We define the event
\begin{equation}\label{3.16}
\wt{E} = \big\{ \o \in \Omega; \; \mbox{for all} \; B \in \cC_\ell^{\rm bound}, \big| B \, \backslash \, \big(\cD_\ell(\o) \cup \cB_\ell(\o)\big)\big| < \mbox{\f $\dis\frac{d-1}{\nu}$} \; (\log \ell) + (\log \ell)^{\beta'}\big\}.
\end{equation}

\n
We will show that for large $\ell$, the event $\wt{E}$ has high probability, and one has the inclusion $\wt{E} \subseteq E$. With this in mind, we first note that as in (4.44), p.~189 of \cite{Szni98a}, for large $\ell$ and any $B \in \cC_\ell$, one has:
\begin{equation}\label{3.17}
\begin{array}{l}
\IP \big[  \big| B \, \backslash \, \big(\cD_\ell(\o) \cup \cB_\ell(\o)\big) \big| \ge \mbox{\f $\dis\frac{d-1}{\nu}$} \;(\log \ell) + (\log \ell)^{\beta'} \big] \le
\\[1ex]
2^{2 \gamma_1^d (\log \log \ell)^d (\log \ell)^\beta} \,\exp\{ - (d-1) \, \log \ell - \nu(\log \ell)^{\beta'}\}.
\end{array}
\end{equation}

\n
Then, by a union bound and the fact that $|\cC_\ell^{\rm bound}| \le 2d \, \ell^{d-1}$ for large $\ell$, we find that
\begin{equation}\label{3.18}
\IP [ \wt{E}^c] \le |\cC_\ell^{\rm bound} | \, 2^{2 \gamma_1^d (\log \log \ell)^d (\log \ell)^{\beta}} \;\ell^{-(d-1)} \,e^{-\nu(\log \ell)^{\beta'}} \le e^{-\frac{\nu}{2} \,(\log \ell)^{\beta'}}, \;\mbox{for large $\ell$}.
\end{equation}
As we now explain,
\begin{equation}\label{3.19}
\mbox{for large $\ell$, $\wt{E} \subseteq E$}.
\end{equation}

\n
Indeed, when $\ell$ is large, then for any $\o \in \wt{E} \backslash E$, one can find $B \in \cC_\ell^{\rm bound}$ such that $\lambda_{1,\o}(B) \le c_0 (\log \ell)^{-2/d} + 100 \gamma_2 (\log \ell)^{-3/d}$, so that with $\rho$ as in Theorem \ref{theo2.1}, $\lambda_{1,\o}(B) + (\log \ell)^{-(2 + \rho)/d} < 2c_0 (\log \ell)^{-2/d}$. Thus by (\ref{2.3}) with $M = 2c_0$, for large $\ell$, for any $\o \in \wt{E} \backslash E$, one can find $B \in \cC_\ell^{\rm bound}$ such that $\lo (B) \le c_0 \lol^{-2/d} + 100 \gamma_2 \,\lol^{-3/d}$ and
\begin{equation}\label{3.20}
\begin{split}
\lambda_{1,\o} (B)  & \ge \lambda_{1,\o} \big(B \, \backslash \, \overline{\cD_\ell(\o)} \big) - (\log \ell)^{-(2 + \rho)/d}
\\[1ex]
&\hspace{-4.1ex} \stackrel{\rm Faber-Krahn}{\ge} \lambda_d \big\{\o_d \, / \, \big| B\, \backslash \, \overline{\cD_\ell(\o)} \big|\big\}^{2/d} - (\log \ell)^{-(2 + \rho) / d} 
\\[1ex]
&\!\!\!\! \stackrel{(\ref{2.1})}{=} \lambda_d \big\{\o_d  \, / \, \big(\big| B \, \backslash \, \big(\cD_\ell(\o) \cup \cB_\ell(\o)\big) | + |B \cap \cB_\ell(\o)|\big)\big\}^{2/d} - (\log \ell)^{-(2 + \rho)/d}
\\[1ex]
&\hspace{-3.7ex} \stackrel{(\ref{3.16}),(\ref{2.11})}{\ge} \lambda_d \big\{\o_d  \, / \, \big(\mbox{\f $\dis\frac{d-1}{\nu}$} \;(\log \ell) + (\log \ell)^{\beta'} + \gamma_1^d(\log \log \ell)^d (\log \ell)^{1- \kappa/d}\big)\big\}^{2/d}  
\\
& \quad - (\log \ell)^{-(2 + \rho)/d}
\\[1ex]
& > \; c_0 (\log \ell)^{-2/d} + 100 \gamma_2 (\log \ell)^{-3/d}, \; \mbox{since} \; c_0 \stackrel{(\ref{0.3})}{=} \lambda_d \{\o_d \, \nu/d\}^{2/d},
\end{split}
\end{equation}
a contradiction. This proves (\ref{3.19}).

\medskip
Combining (\ref{3.18}) and (\ref{3.19}), the claim (\ref{3.15}) follows. This proves Lemma \ref{lem3.1}.
\end{proof}

We now introduce (see above (4.56), p.~192 of \cite{Szni98a}):
\begin{equation}\label{3.21}
\begin{split}
\wt{\cC}_\ell = & \;\mbox{the collection of boxes $\wt{B} = (\log \ell)^{1/d} \big(q + (0,2 [ \gamma_1 \log \log \ell])^d\big)$},
\\[-0.5ex]
&\;\mbox{$q \in \IZ^d$ that intersect $B_\ell$}.
\end{split}
\end{equation}

\n
(The side-length of the boxes in $\wt{\cC}_\ell$ is the double of that of the boxes in $\cC_\ell$, see (\ref{3.6}).)

\medskip
Then, we consider the event, see (4.56) on p.~192 of \cite{Szni98a} (where we have chosen $\zeta = \frac{3}{4}$):
\begin{equation}\label{3.22}
\begin{split}
H = & \; \big\{ \o \in \Omega: \;\mbox{for all $B, B'$ in $\cC_\ell$ with $B \cap B' = \emptyset$ and ${\rm diam} (B \cup B') \le \ell^{3/4} (\log \ell)^{1/d}$},
\\[-0.5ex]
& \;\; \;|B \backslash \big(\cD_\ell(\o) \cup \cB_\ell(\o)\big)| + \big| B' \backslash (\cD_\ell(\o) \cup \cB_\ell(\o)\big)\big| \le \mbox{\f $\dis\frac{d}{\nu}$} \; \big(1 + \mbox{\f $\dis\frac{3}{4}$}\big) (\log \ell) + (\log \ell)^{\beta'}
\\[-0.5ex]
&\quad \mbox{and for all} \; \wt{B} \in \wt{\cC}_\ell, \, \big|\wt{B} \backslash \big(\cD_\ell(\o) \cup \cB_\ell(\o)\big)\big| \le  \mbox{\f $\dis\frac{d}{\nu}$} \; (\log \ell) + (\log \ell)^{\beta'}\;
\end{split}
\end{equation}

\n
It then follows from Lemma 4.8 on p.~192 of \cite{Szni98a}  that
\begin{equation}\label{3.23}
\mbox{for large $\ell$}, \IP[H]  \ge 1 - \exp \big\{ -\mbox{\f $\dis\frac{\nu}{2}$} \;(\log \ell)^{\beta'}\big\}.
\end{equation}

\n
(We could replace $\frac{3}{4}$ by any $\zeta \in (0,1)$ arbitrarily close to $1$ in the definition (\ref{3.22}) with a corresponding (\ref{3.23}), but the choice $\zeta = \frac{3}{4}$ will suffice for our purpose).

\medskip
We now define the event encapsulating the nature of the typical configurations $\o$, which we will consider in the analysis of $\lambda_{1,\o}(B_\ell)$ and $\lambda_{2,\o}(B_\ell)$, in the next two sections. We recall (\ref{3.7}), (\ref{3.9}), (\ref{3.11}), (\ref{3.13}), (\ref{3.22}) and set
\begin{equation}\label{3.24}
T = C \cap E \cap F \cap G \cap H,
\end{equation}
so that by (\ref{3.8}), (\ref{3.10}), (\ref{3.12}), (\ref{3.15}), (\ref{3.23}), 
\begin{equation}\label{3.25}
\mbox{for large $\ell$}, \IP[T] \ge 1 - 5 \, \exp\big\{ -\mbox{\f $\dis\frac{\nu}{2}$} \;  (\log \ell)^{\beta'}\big\}.
\end{equation}

\section{Localization}
\setcounter{equation}{0}

We now proceed with the investigation of the first two eigenvalues $\lambda_{1,\o}(B_\ell)$, $\lambda_{2,\o}(B_\ell)$, see (\ref{1.5}), (\ref{1.9}), when $\ell$ is large, under the occurrence of the likely event $T$, see (\ref{3.24}). Whereas $\lambda_{1,\o}(B_\ell)$ is comparable to $c_0(\log \ell)^{-2/d}$ on $T$ (with $c_0$ from (\ref{0.3})), we introduce a certain ``resonance event'' $R$ in (\ref{4.5}), where $\lambda_{2,\o}(B_\ell) < \lambda_{1,\o} (B_\ell) + \sigma(\log \ell)^{-(1+2/d)}$ (with fixed $\sigma > 0$). The main Theorem \ref{theo4.1} of this section shows that for large $\ell$ on $T \cap R$, one can find two distant sub-boxes of side-length of order $(\log \ell)^{1/d}$ contained in $B_\ell$, with principal Dirichlet eigenvalues, which are close to $\lambda_{1,\o}(B_\ell)$ in scale $(\log \ell)^{-(1 + 2/d)}$. In addition, the principal Dirichlet eigenfunctions attached to these boxes are well localized in balls of radius $R_0(\log \ell)^{1/d} (1 + o(1))$ with centers close to the respective centers of these boxes. These results will in essence follow from the application of the method of enlargement of obstacles recalled in Section 2, and the quantitative Faber-Krahn inequality (\ref{1.17}). The reduction to the analysis of boxes with size of order $(\log \ell)^{1/d}$ (unlike the boxes in $\cC_\ell$, see (\ref{3.6})) will be important in Section 6 for the quality of the deconcentration estimates in Theorem \ref{theo6.2}. We refer to the beginning of Section 3 concerning the choices of parameters (notably for the method of enlargement of obstacles), which remain in force, and for the convention concerning positive constants.

\medskip
We first need some additional notation. We introduce the length
\begin{equation}\label{4.1}
L_0 = 10 (\lceil R_0 \rceil + 1) (\log \ell)^{1/d} \; \mbox{(with $R_0$ from (\ref{0.2}) and $\ell > 10$ throughout)},
\end{equation}
and consider the open boxes of side-length $L_0$ and center in $(\log \ell)^{1/d} \,\IZ^d$:
\begin{equation}\label{4.2}
D_{0,q} = q (\log \ell)^{1/d} + \big( - \mbox{\f $\dis\frac{L_0}{2}$}, \mbox{\f $\dis\frac{L_0}{2}$}\big)^d, \; q \in \IZ^d.
\end{equation}

\n
We will typically write $D_0$ or speak of an {\it $L_0$-box} to refer to a generic box of the form $D_{0,q}$. Given such a box $D_0$, we will refer to 
\begin{equation}\label{4.3}
\begin{array}{l}
\mbox{{\it the central box} of $D_0$ to denote the closed box of side-length $2(\log \ell)^{1/d}$}
\\
\mbox{and same center as $D_0$}.
\end{array}
\end{equation}
To define the resonance set, we first pick
\begin{equation}\label{4.4}
\sigma > 0.
\end{equation}

\n
We will eventually let $\sigma$ tend to $0$ in Section 6. The {\it resonance event} is then
\begin{equation}\label{4.5}
\mbox{$R = \{ \o \in \Omega; \lambda_{1,\o}(B_\ell) < \infty$ and $\lambda_{2,\o}(B_\ell) < \lambda_{1,\o}(B_\ell) + \rho_\ell\}$, where $\rho_\ell = \sigma (\log \ell)^{-(1 + 2/d)}$}.
\end{equation}

\n
We will show in Theorem \ref{theo6.1} that $\lim_{\sigma \r 0} \limsup_{\ell \r \infty} \IP[R] = 0$ (this will be the lower bound on the spectral gap). For the time being the main object of this section is the proof of Theorem \ref{theo4.1}. The simpler Theorem \ref{theo4.2} is also of interest. We recall from (\ref{3.9}), (\ref{4.5}) that $\delta_\ell = (\log \ell)^{-(2+ 2/d)}$, $\rho_\ell = \sigma(\log \ell)^{-(1 + 2/d)}$, and $\gamma_2$ is the constant from (\ref{3.9}). The likely event $T$ is defined in (\ref{3.24}).

\begin{theorem}\label{theo4.1}
There exist $\wh{\eta}_1, \wh{\eta}$ in $(0,\frac{1}{d})$ and $\ell_0 \ge 10$ such that for $\ell \ge \ell_0$, on the event $T \cap R$, one has boxes $D_0,D_0'$ contained in $B_\ell$ with $d_\infty (D_0, D'_0) \ge \ell^{3/4}$ such that
\begin{equation}\label{4.6}
\left\{ \begin{array}{l}
\lambda_{1,\o} (B_\ell) \le \lambda_{1,\o}(D_0) \le \lambda_{1,\o} (B_\ell) + \rho_\ell + \delta_\ell + e^{-(\log \ell)^{\wh{\eta}}} \le c_0 (\log \ell)^{-2/d} + 2 \gamma_2 (\log \ell)^{-3/d}, 
\\
\mbox{with similar inequalities for $\lambda_{1,\o}(D'_0)$}.
\end{array}\right.
\end{equation}
and there are open balls $\wh{B}, \wh{B}'$ with centers having rational coordinates, belonging to the respective central boxes (see (\ref{4.3})) of $D_0$ and $D'_0$ with radius $\wh{R} = R_0 (\log \ell)^{1/d} + 2 e^{-(\log \ell)^{\wh{\eta}_1}}$ (and $R_0$ from (\ref{0.2})), such that
\begin{equation}\label{4.7}
\left\{ \begin{array}{l}
\lambda_{1,\o} (D_0) \le \lambda_{1,\o}(\wh{B}) \le \lambda_{1,\o} (D_0) + e^{-(\log \ell)^{\wh{\eta}}}, 
\\
\mbox{with similar inequalities for $\lambda_{1,\o}(D'_0), \lambda_{1,\o}(\wh{B}')$},
\end{array}\right.
\end{equation}
and (with the notation (\ref{1.14}))
\begin{equation}\label{4.8}
\left\{ \begin{array}{l}
\varphi_{1,D_0,\o} \le e^{-(\log \ell)^{\wh{\eta}}} \; \mbox{on} \; D_0 \backslash \wh{B},
\\
\mbox{with a similar inequality for $\varphi_{1,D'_0,\o}$ on $D'_0 \backslash \wh{B}$}.
\end{array}\right.
\end{equation}
\end{theorem}

We stated (\ref{4.7}) for completeness but our main interest in view of Sections 5 and 6 lies in (\ref{4.8}). The rational coordinates of the centers of the balls are mentioned to highlight the measurability of the events under consideration. One also has the simpler

\begin{theorem}\label{theo4.2}
With $\wh{\eta}_1, \wh{\eta}$ as in Theorem \ref{theo4.1}, there exists $\ell_1 >10$ such that for $\ell \ge \ell_1$, on the event $T$, one has a box $D_0$ contained in $B_\ell$ such that
\begin{equation}\label{4.9}
\lob \le \lo(D_0) \le \lob +\delta_\ell + e^{-\lol^{\wh{\eta}}}\le c_0 \lol^{-2/d}  + 2 \gamma_2\lol^{-3/d},
\end{equation}
and an open ball $B^\#$ with center having rational coordinates, belonging to the central box of $D_0$ with radius $\wh{R}$, such that the first lines of (\ref{4.7}) and (\ref{4.8}) hold with $B^\#$ in place of $\wh{B}$.
\end{theorem}

The proof of Theorem \ref{theo4.2} is simpler than that of Theorem \ref{theo4.1} and is also quite similar. We mainly focus on the proof of Theorem \ref{theo4.1} in the remainder of this section. In Remark \ref{rem4.7} 1) we briefly sketch the main steps in the proof of Theorem \ref{theo4.2}.

\medskip
It may be appropriate to describe here the general line of the arguments, which we use. We begin with two lemmas, which for large $\ell$ on $T \cap R$ provide us with two distant boxes $B, B'$ in $\cC_\ell^{\rm int}$ (see (\ref{3.14})) with principal Dirichlet eigenvalues, essentially within $\rho_\ell$ ($= \sigma \lol^{-(1 + 2/d)}$) from $\lob$, and such that after deletion of the closure of the density set, the corresponding principal Dirichlet eigenvalues do not increase too much (but may well be much bigger than $\lob + \rho_\ell)$. In the Proposition \ref{prop4.5} we combine the volume estimates and the eigenvalue estimates for $B \backslash \overline{\cD_\ell(\o)}$ and $B' \backslash \overline{\cD_\ell(\o)}$ with the quantitative Faber-Krahn inequality (\ref{1.17}) to bring into play two balls of same respective volumes as the above two sets, and symmetric differences with these respective sets of small volume. We also find two suitable boxes of $\wt{\cC}_\ell$ (see (\ref{3.21})) each containing one of the above balls close to their center, and such that after deletion of the ball and $\cD_\ell(\o)$, the volume of the remaining set in each box is small compared to $\log \ell$. Once Proposition \ref{prop4.5} is proved, we can apply Theorem \ref{theo2.3}, and also use the representation formula (\ref{1.12}) for eigenfunctions combined with Lemma \ref{lem2.2}, to establish the existence of two distant $L_0$-boxes (concentric with the above boxes of $\wt{\cC}_\ell$) having the desired principal Dirichlet eigenvalue estimates, and adequately small principal Dirichlet eigenfunctions outside balls of deterministic radius $\wh{R}$ (see above (\ref{4.7})) with suitable centers in the central boxes (see (\ref{4.3})) of these $L_0$-boxes.

\medskip
With this plan in mind, the first step is a deterministic statement.
\begin{lemma}\label{lem4.3}
For $\ell > 10$ and $\o \in \Omega$,
\begin{equation}\label{4.10}
\begin{array}{l}
\mbox{when $\lob < \infty$, one can find disjoint open subsets $U, U'$ in $B_\ell$ such that}
\\
\mbox{$\lo(U)$ and $\lo(U')$ belong to the interval $[\lob, \llob]$}.
\end{array}
\end{equation}
\end{lemma}

\begin{proof}
When $\lob < \infty$ (or equivalently when $B_{\ell,\o} \not= \emptyset$, see (\ref{1.4})), as noted below (\ref{1.9}), the eigenvalues $\lambda_{i,\o}, i \ge 1$, correspond to the reordering of the union (with multiplicities) of the Dirichlet eigenvalues of $-\frac{1}{2} \, \Delta$ in each of the finitely many connected components of $B_{\ell, \o}$. Thus, at least one of the items below occurs:
\begin{equation}\label{4.11}
\left\{ \begin{array}{rl}
{\rm i)} & \mbox{$\lob$ and $\llob$ correspond to principal Dirichlet eigenvalues}
\\
&\mbox{of $-\frac{1}{2} \, \Delta$ in distinct connected components of $B_{\ell,\o}$},
\\[2ex]
{\rm ii)} & \mbox{$\lob$ and $\llob$ are the first two Dirichlet eigenvalues }
\\
&\mbox{of $-\frac{1}{2} \, \Delta$ in one of the connected components of $B_{\ell,\o}$}.
\end{array}\right.
\end{equation}

\n
When (\ref{4.11}) i) occurs, the claim (\ref{4.10}) is immediate: one simply chooses $U$ and $U'$ as the connected components mentioned in (\ref{4.11}) i).

\medskip
We thus assume that (\ref{4.11}) ii) occurs. We denote by $W$ a connected component of $B_{\ell,\o}$ such that the first two Dirichlet eigenvalues of $-\frac{1}{2} \, \Delta$ in $W$ respectively coincide with $\lob$ and $\llob$ and by $\psi$ an $L^2$-normalized Dirichlet eigenfunction in $W$ corresponding to $\llob$. Note that $W$ satisfies an exterior cone condition (each boundary point of $W$ belongs to $B^c_\ell$ or to a closed ball of radius $a$ in $W^c$). As explained below (\ref{1.10}), the function $\psi$ is continuous and it equals $0$ outside $W$. Since $\psi$ is attached to the second Dirichlet eigenvalue of $-\frac{1}{2} \, \Delta$ (and orthogonal to the principal Dirichlet eigenfunction on the connected open set $W$), it changes sign, and we can choose non-empty connected components $U$ of $\{\psi > 0\}$ and $U'$ of $\{\psi < 0\}$. As we now explain
\begin{equation}\label{4.12}
\mbox{$\lo(U)$ and $\lo(U')$ are equal to $\llob$}.
\end{equation}

\n
This will complete the proof of (\ref{4.10}). We now prove (\ref{4.12}) and consider the case of $U$. The eigenfunction $\psi$ satisfies $- \frac{1}{2}\, \Delta \psi = \llob \psi$ in $U$ (in a classical sense, see for instance Theorem 11.7, p.~279 of \cite{LiebLoss01}). As we now explain $\psi 1_U$ belong to $H^1_0(U)$. Indeed, for $\alpha > 0, (\psi - \alpha)_+ 1_U$ is compactly supported in $U$, and has gradient equal to $\nabla \psi 1_{U \cap \{\psi > \alpha\}}$, see Corollary 6.18, p.~153 of \cite{LiebLoss01}. As $\alpha$ tends to $0$, $(\psi - \alpha)_+ 1_U$ is thus Cauchy in $H^1_0(U)$ and converges in $L^2(\IR^d)$ to $\psi 1_U$, so that $\psi 1_U \in H^1_0 (U)$.

\medskip
We write $\varphi$ as a shorthand for $\varphi_{1,U,\o}$, see (\ref{1.14}). Consider now two sequences $\psi_n$ and $\varphi_n$ of smooth functions compactly supported in $U$ respectively converging in $H^1_0(U)$ to $\psi 1_U$ and $\varphi$. Then for each $n \ge 1$, we have
\begin{equation}\label{4.13}
\begin{array}{rl}
{\rm i)}& \llob \dis\int \psi \varphi_n \,dx = \dis\int - \fr  \, \Delta \, \psi \varphi_n \, dx \stackrel{\rm int.\;by\;parts}{=} \fr \; \dis\int \nabla \psi \, \nabla \varphi_n \, dx.
\\[2ex]
{\rm ii)} &\lambda_{1,\o}(U) \dis\int \varphi \psi_n \,dx = \dis\int - \fr \, \Delta \varphi \psi_n\, dx \stackrel{\rm int.\;by\;parts}{=} \fr \; \dis\int \nabla \varphi \, \nabla \psi_n \, dx.
\end{array}
\end{equation}

\n
Thus, letting $n$ tend to infinity, we find that the right expressions in i) and ii) both converge to $\int_U \nabla \psi \nabla \varphi \,dx$ and the left expressions respectively converge to $\lambda_{2,\o}(B_\ell) \,\int \psi \varphi \, dx$ and $\lambda_{1,\o}(U) \int \varphi \psi \,dx$. The integral being positive, this shows that $\lo(U) = \llob$. In a similar fashion one has $\lo (U') = \llob$ and (\ref{4.12}) follows. This completes the proof of Lemma \ref{lem4.3}.
\end{proof}

We recall that $\rho_\ell = \sigma(\log \ell)^{-(1 + 2/d)}$, see (\ref{4.5}), and $\delta_\ell = (\log \ell)^{-(2 + 2/d)}$, see (\ref{3.9}). The next step towards the proof of Theorem \ref{theo4.1} is

\begin{lemma}\label{lem4.4}
For large $\ell$, on $T \cap R$ there exist $B, B'$ in $\cC^{\rm int}_\ell$ (see (\ref{3.14})) such that
\begin{equation}\label{4.14}
d_\infty (B,B') \ge \mbox{\f $\dis\frac{1}{2d}$} \; \ell^{3/4} (\log \ell)^{1/d},
\end{equation}
\begin{equation}\label{4.15}
\left\{ \begin{array}{rl}
{\rm i)} & \lob \le \lo (B) \le \lob + \rho_\ell + \delta_\ell \le c_0 \lol^{-2/d} + 2 \gamma_2 \lol^{-3/d},
\\[2ex]
{\rm ii)} & \lo \big(B \backslash \overline{\cD_\ell (\o)}\big) \le \lo (B) + \lol^{-(2 + \rho)/d},
\\[2ex]
&\hspace{-4.5ex} \mbox{and similar inequalities as {\rm i)} and {\rm ii)} with $B'$ in place of $B$.}
\end{array}\right.
\end{equation}
(With $\gamma_2$ as in (\ref{3.9}) and $\rho$ as in Theorem \ref{theo2.1}).
\end{lemma}

\begin{proof}
We first observe that for large $\ell$ on $T \cap R$ we have by (\ref{3.9}) and (\ref{4.5})
\begin{equation}\label{4.16}
\llob \le c_0 \lol^{-2/d} + \gamma_2 \lol^{-3/d} + \rho_\ell \le c_0 \lol^{-2/d} + \mbox{\normalsize $\frac{3}{2}$}\; \gamma_2 \lol^{-3/d}
\end{equation}

\n
and by Lemma \ref{lem4.3} there are disjoint open subsets $U, U'$ of $B_\ell$ such that $\lo (U)$ and $\lo (U')$ are at most $\llob < 2 c_0 \lol^{-2/d} - \delta_\ell$.

\medskip
Then, as in (4.54) on p.~192 of \cite{Szni98a} (using the fact that $C \subseteq T$, see (\ref{3.7}), (\ref{3.24}), and Theorem \ref{theo2.3} with $M = 2c_0$), one can find boxes $B, B'$ in $\cC_\ell$ such that
\begin{equation}\label{4.17}
\left\{ \begin{array}{l}
\llob \ge \lo (U) \ge \lo (U \cap B) - \delta_\ell,
\\[2ex]
\llob  \ge \lo (U') \ge \lo (U' \cap B') - \delta_\ell.
\end{array}\right.
\end{equation}

\n
In addition, since $\lo (B)$ and $\lo (B')$ are at most $\llob + \delta_\ell \le c_0 \lol^{-2/d} + 2 \gamma_2 \lol^{-3/d}$, by (\ref{4.16}), and keeping in mind that $T \subseteq E$, see (\ref{3.13}), (\ref{3.14}), we can additionally assume ($\ell$ being large enough) that
\begin{equation}\label{4.18}
\mbox{$B, B'$ belong to $\cC^{\rm int}_\ell$ (and in particular are included in $B_\ell$)}.
\end{equation}
This already proves that
\begin{equation}\label{4.19}
\mbox{the statement (\ref{4.15}) i) and the corresponding statement for $B'$ hold}.
\end{equation}

\n
Next, as a result of Theorem \ref{theo2.1} (with $M = 2c_0$), when $\ell$ is large so that in particular $c_0\lol^{-2/d} + \rho_\ell + \delta_\ell + \lol^{-(2 + \rho)/d} < 2 c_0 \lol^{-2/d}$, on $T \cap R$ in addition to (\ref{4.16}) - (\ref{4.19}), one has
\begin{equation}\label{4.20}
\begin{split}
\lo (U \cap B) & \ge \lo \big((U \cap B) \, \backslash \, \overline{\cD_\ell(\o)} \,\big) \wedge \big(2 c_0 \lol^{-2/d}\big) - \lol^{-(2 + \rho)/d}
\\[1ex]
& = \lo  \big((U \cap B) \, \backslash \, \overline{\cD_\ell(\o)} \,\big) -  \lol^{-(2 + \rho)/d}
\\[1ex]
&\hspace{-4.3ex} \stackrel{\rm Faber-Krahn}{\ge} \lambda_d \big\{ \o_d \, / \, |(U \cap B)  \, \backslash \, \overline{\cD_\ell(\o)} \,|\big\}^{2/d} - \lol^{-(2 + \rho)/d}
\\[1ex]
& \ge  \lambda_d \big\{ \o_d \, / \, \big[\big|(U \cap B)  \, \backslash \, (\cD_\ell(\o) \cup \cB_\ell(\o)\big)\big| + | B \cap \cB_\ell(\o)|\big]\big\}^{2/d} 
\\[1ex]
&\quad -  \lol^{-(2 + \rho)/d}
\end{split}
\end{equation}
Then, making use of Theorem \ref{theo2.4} and (\ref{3.6}) to bound $|B \cap \cB_\ell(\o)|$, and arguing in the same fashion for $U'$ and $B'$, we see that for large $\ell$ on $T \cap R$, one has in addition to (\ref{4.16}) - (\ref{4.19})
\begin{equation}\label{4.21}
\left\{ \begin{array}{l}
\lo (U \cap B) \ge \lambda_d\big\{\o_d \, / \, \big[\big|(U \cap B) \, \backslash \, \big(\cD_\ell (\o) \cup \cB_\ell(\o)\big)\big| +
\\[1ex]
\qquad \qquad \qquad \;   \gamma_1^d (\log \log \ell)^d \,  \lol^{1- \kappa/d}\big]\big\}^{2/d} - \lol^{-(2 + \rho)/d}
\\[1ex]
\mbox{and similar inequalities with $U', B'$ in place of $U,B$}.
\end{array}\right.
\end{equation}

\medskip\n
Now, both $\lo (U \cap B)$ and $\lo (U' \cap B')$ are at most $c_0 \lol^{-2/d} + \gamma_2 \lol^{-3/d} + \rho_\ell + \delta_\ell$ by (\ref{4.17}), (\ref{4.16}), and $c_0 = \lambda_d\{\o_d \, \nu/d\}^{2/d}$ by (\ref{0.3}). It now follows from this observation and (\ref{4.21}) that both $|(U \cap B) \backslash (\cD_\ell(\o) \cup \cB_\ell(\o))|$ and $|(U' \cap B') \backslash (\cD_\ell(\o) \cup \cB_\ell(\o))|$ exceed $\frac{d}{\nu} \, (1 - c_\ell) \, \log \ell$ where $c_\ell$ is a deterministic function of $\ell$ tending to $0$ as $\ell$ goes to infinity. Moreover, $U$ and $U'$ are disjoint, and if $B \cap B' \not= \emptyset$ then $B \cap B'$ is contained in some $\wt{B} \in \wt{\cC}_\ell$ for which
\begin{equation}\label{4.22}
\begin{array}{l}
\big| \wt{B} \, \backslash \, \big( \cD_\ell(\o) \cup \cB_\ell(\o)\big)\big| \ge \big|(U \cap B)  \, \backslash \,\big( \cD_\ell(\o) \cup B_\ell(\o)\big)\big|  \;+ 
\\[1ex]
\big| (U' \cap B') \, \backslash \, ( \cD_\ell(\o) \cup \cB_\ell(\o)\big)\big| \ge \mbox{\f $\dis\frac{2d}{\nu}$} \;(1 - c_\ell).
\end{array}
\end{equation}
On the other hand, it follows from the inclusion $T \subseteq H$, see (\ref{3.22}), (\ref{3.24}), that the left member of (\ref{4.22}) is at most $\frac{d}{\nu}\,\lol + \lol^{\beta '}$. So when $\ell$ is large, we can in addition assume that $B \cap B' = \emptyset$ and coming back to (\ref{3.22}) that ${\rm diam} (B \cup B') \ge \ell^{3/4} \lol^{1/d}$ so that (\ref{4.14}) holds.

\medskip
The claim (\ref{4.15}) ii) and the corresponding claim for $B'$ now follow from the application of Theorem \ref{theo2.1}, the bound $\lo (B) + \lol^{-(2 + \rho)/d} < 2 c_0 \lol^{-2/d}$ and the similar bound for $B'$, which are consequences of (\ref{4.15}) i) and the corresponding bound for $B'$ ($\ell$ being large). This concludes the proof of Lemma \ref{lem4.4}.
\end{proof}

We will now gather upper bounds on eigenvalues and on volume, and combine them with the quantitative Faber-Krahn inequality (\ref{1.17}) in the course of the proof of the next proposition. We recall the definition of $\wt{\cC}_\ell$ in (\ref{3.21}), and we define the central box of $\wt{B}_\ell$ in $\wt{\cC}_\ell$ similarly as in (\ref{4.3}), namely as the closed concentric box of $\wt{B}_\ell$ with side-length $2 \lol^{1/d}$. The parameter $\alpha$ appeared in the selection made for the method of enlargement of obstacles, see Lemma \ref{lem2.2} and Remark \ref{rem4.6} below.

\begin{proposition}\label{prop4.5}
There exist $\wh{\mu} \in (1 - \alpha, 1)$ and $\wh{\eta}_0 < \wh{\eta}_1$ in $(\frac{\wh{\mu}}{d}, \frac{1}{d})$ such that for large $\ell$ on $T \cap R$, there are two boxes $\wt{B}$ and $\wt{B}'$ in $\wt{\cC}_\ell$, included in $B_\ell$, with $d_\infty (\wt{B}, \wt{B}') \ge \ell^{3/4}$, which respectively contain open balls $\wh{B}$ and $\wh{B}'$ with centers having rational coordinates that belong to the respective central boxes of $\wt{B}$ and $\wt{B}'$, and have same radius $\wh{R} = R_0 \lol^{1/d} + 2 \lol^{\wh{\eta}_1}$. These balls have the property that denoting by $\wh{B}_{\rm int}$ and $\wh{B}'_{\rm int}$ the smaller closed concentric balls with radius $\wh{R}_{\rm int} = R_0 \lol^{1/d} + \lol^{\wh{\eta}_0}$, one has
 \begin{equation}\label{4.23}
\left\{ \begin{array}{l}
| \wt{B} \, \backslash \, (\cD_\ell (\o) \cup \wh{B}_{\rm int} )| \le \lol^{\wh{\mu}},
\\[1ex]
\mbox{and a similar inequality with $\wt{B}'$ and $\wh{B}'_{\rm int}$ in place of $\wt{B}$ and $\wh{B}_{\rm int},$~~~~}
\end{array}\right.
\end{equation}
as well as (in the notation of (\ref{4.5}) and (\ref{3.9}))
\begin{equation}\label{4.24}
\left\{ \begin{array}{l}
\lob \le \lo (\wt{B}) \le \lob + \rho_\ell + \delta_\ell \le c_0 \lol^{-2/d} + 2 \gamma_2 \lol^{-3/d},
\\[1ex]
\mbox{and similar inequalities with $\wt{B}'$  in place of $\wt{B}$.}
\end{array}\right.
\end{equation}
\end{proposition}

\begin{proof}
We first choose (recall that $\beta', \kappa, \rho$ are among the parameters selected at the beginning of Section 3)
\begin{equation}\label{4.25}
\mu' \in\big(\max \big(\beta', 1- \mbox{\f $\dis\frac{\kappa}{d}$}\big), 1\big) \;\mbox{and} \; \chi' \in \big(0, \min \big(1- \mu', \, \mbox{\f $\dis\frac{1 \wedge \rho}{d}$}\big)\big).
\end{equation}

\n
Then by Lemma \ref{lem4.4}, for large $\ell$, on $T \cap R$ we have two boxes $B, B' \in \cC^{\rm int}_\ell$, which satisfy (\ref{4.14}), (\ref{4.15}). By the inclusion $G \subseteq T$, see (\ref{3.11}), (\ref{3.24}), and the volume bound on the bad set from Theorem \ref{theo2.4}, we can further assume that
\begin{equation*}
\begin{split}
\big| B \, \backslash \, \overline{\cD_\ell(\o)} \,\big| \le & \; \big| B \, \backslash \, \big(\cD_\ell(\o) \cup \cB_\ell(\o)\big)\big| + | B \cap \cB_\ell(\o)| \le \mbox{\f $\dis\frac{d}{\nu}$} \; \lol + \lol^{\beta'} +
\\
& \; \gamma^d_1 (\log \log)^d \lol^{1-\kappa/d},
\end{split}
\end{equation*}
and a similar bound for $| B' \backslash \overline{\cD_\ell(\o)} |$.

\medskip
So, by our choice of $\mu'$ in (\ref{4.25}), we can further assume that
\begin{equation}\label{4.26}
\mbox{$| B \, \backslash \, \overline{\cD_\ell(\o)} | \le  \mbox{\f $\dis\frac{d}{\nu}$} \;\lol + \lol^{\mu '}$ and a similar inequality with $B'$ in place of $B$}.
\end{equation}
By (\ref{4.15}) i) and ii) we additionally know that
\begin{equation}\label{4.27}
\begin{array}{l}
\lo \big(B \, \backslash \, \overline{\cD_\ell(\o)} \big) \le c_0 \lol^{-2/d}  + 2 \gamma_2 \lol^{-3/d} + \lol^{-(2 + \rho)/d}
\\[1ex]
\mbox{and a similar inequality with $B'$ in place of $B$}.
\end{array}
\end{equation}

\begin{samepage}
\n
The eigenvalues in (\ref{4.27}) are bigger or equal to $\lambda_{-\frac{1}{2}\, \Delta} (B \backslash \overline{\cD_\ell(\o)})$ and $\lambda _{-\frac{1}{2}\, \Delta} (B' \backslash \overline{\cD_\ell(\o)})$ respectively. With $c_0 = \lambda_d (\nu \o_d / d)^{2/d}$, see (\ref{0.3}), and the choice of $\chi'$ in (\ref{4.25}), we can thus assume that for large $\ell$ on $T \cap R$ we have $B, B'$ in $\cC^{\rm int}_\ell$ satisfying (\ref{4.14}), (\ref{4.15}), (\ref{4.26}), (\ref{4.27}) and
\begin{equation}\label{4.28}
\begin{array}{l}
0 \le \big\{\lambda_{-\frac{1}{2}\, \Delta}  \big(B \, \backslash \, \overline{\cD_\ell(\o)} \big) \, /\, \lambda_d \big\} \; \big\{ \big| B \, \backslash \, \overline{\cD_\ell(\o)}\big| \, / \, \o_d\big\}^{2/d} - 1 \le \lol^{-\chi'}
\\[1ex]
\mbox{and similar inequalities with $B'$ in place of $B$}.
\end{array}
\end{equation}
(The first inequality in (\ref{4.28}) results from the Faber-Krahn inequality).
\end{samepage}

\medskip
By the quantitative Faber-Krahn inequality (\ref{1.17}) one can then find balls $\check{B}$ and $\check{B}'$ with centers having rational coordinates, with same respective volumes as $|B\backslash \overline{\cD_\ell(\o)}|$ and $|B'\backslash \overline{\cD_\ell(\o)}|$ and so that the ratio of the volumes of the symmetric differences $(B\backslash \overline{\cD_\ell(\o)}) \, \Delta \, \check{B}$ and $(B' \backslash \overline{\cD_\ell(\o)}) \, \Delta \, \check{B}'$ to the respective volumes $| \check{B}|$ and $| \check{B}'|$ is smaller than  $2 / \sqrt{c}_2  \linebreak \lol^{-\chi'/2}< \frac{1}{2}$. In particular, the centers of $\check{B}$ and $\check{B}'$ respectively belong to $B$ and $B'$ (otherwise the above mentioned ratios would be at least $\frac{1}{2}$).

\medskip
In addition, by (\ref{4.27}) and the Faber-Krahn inequality, due to the value of $c_0$ recalled above (\ref{4.28}), and the choice $\chi' < (1 \wedge \rho)/d$ in (\ref{4.25}), we can additionally assume that
\begin{equation}\label{4.29}
\begin{array}{l}
\big|  B \, \backslash \, \overline{\cD_\ell(\o)}  \big|  \ge \o_d \, \lambda_d^{d/2} \; \lambda_{-\frac{1}{2} \, \Delta} \big(B \,\backslash \,\overline{\cD_\ell(\o)}\big)^{-d/2} \stackrel{(\ref{4.27})}{\ge} \mbox{\f $\dis\frac{d}{\nu}$} \;\lol - \lol^{1- \chi'}
\\[1ex]
\mbox{and a similar inequality with $B'$ in place of $B$}.
\end{array}
\end{equation}

\n
Thus, combining (\ref{4.26}) and (\ref{4.29}), we have upper and lower bounds on the volumes of $B \backslash \overline{\cD_\ell(\o)}$ and $B' \backslash \overline{\cD_\ell(\o)}$ which respectively coincide with $|\check{B}|$ and $|\check{B}'|$. By the bound on the volumes of the symmetric differences stated below (\ref{4.28}), we see that, $\ell$ being large, we can assume that
\begin{equation*}
\begin{split}
\big| \big(B \, \backslash \, \overline{\cD_\ell(\o)} \big) \cap \check{B}\big| & \ge | \check{B} | \,\big(1 -  2 / \sqrt{c}_2 \,\lol^{-\chi'/2}\big)
\\[1ex]
&\hspace{-4ex} \stackrel{(\ref{4.29}),(\ref{4.26})}{\ge} \; \mbox{\f $\dis\frac{d}{\nu}$} \; \lol - \lol^{1 - \chi'} - c \lol^{1- \chi'/2},
\end{split}
\end{equation*}
with similar inequalities for $B'$, $\check{B}'$, so that
\begin{equation}\label{4.30}
\left\{ \begin{array}{l} 
\big| \big(B \, \backslash \, \overline{\cD_\ell(\o)} \big) \cap \check{B} \big|  \ge \mbox{\f $\dis\frac{d}{\nu}$} \; \lol - \ov{c} \,\lol^{1- \chi'/2}
\\[1ex]
\mbox{with a similar lower bound with $B'$ and $\check{B}'$ in place of $B$ and $B'$}.
\end{array}\right.
\end{equation}
We have thus established that for large $\ell$ on $T \cap R$ there are boxes $B$, $B'$ in $\cC_\ell^{\rm int}$, which satisfy (\ref{4.14}), (\ref{4.15}) and balls $\check{B}$ and $\check{B}'$ with centers having rational coordinates respectively belonging to $B$ and $B'$ with same respective volumes as $B \backslash \overline{\cD_\ell(\o)}$ and $B' \backslash \overline{\cD_\ell(\o)}$, which are at most $\frac{d}{\nu} \, \lol + \lol^{\mu'}$, see (\ref{4.26}), at least $\frac{d}{\nu} \, \lol - \lol^{1- \chi'}$, see (\ref{4.29}), and so that (\ref{4.30}) holds.

\medskip
We can now find boxes $\wt{B}, \wt{B}' \in \wt{\cC}_\ell$, see (\ref{3.21}), such that the centers of $\check{B}$ and $\check{B}'$ belong to the respective central boxes of $\wt{B}, \wt{B}'$, and so that $B \subseteq \wt{B}$, $B' \subseteq \wt{B}'$. Further, $\ell$ being large, we can assume that $\check{B} \subseteq \wt{B}$ and $\check{B}' \subseteq \wt{B}'$, and since $B, B' \in \cC_\ell^{\rm int}$ satisfy (\ref{4.14}), (\ref{4.15} i) that in addition
\begin{equation}\label{4.31}
\mbox{$\wt{B}, \wt{B}' \subseteq B_\ell$ satisfy (\ref{4.24}) and $d_\infty (\wt{B}, \wt{B}') \ge \mbox{\f $\dis\frac{1}{3d}$} \;\ell^{3/4} \, \lol^{1/d}$}.
\end{equation}
As we now explain, there is little volume in $\wt{B}$ outside $\check{B} \cup \cD_\ell(\o)$ and in $\wt{B}'$ outside $\check{B}' \cup \cD_\ell(\o)$ (the remaining claim (\ref{4.23}) will quickly follow). Indeed, for large $\ell$ on $T \cap R$, we can further assume that
\begin{equation}\label{4.32}
\begin{split}
\big| \wt{B}\, \backslash \, \big(\check{B}  \cup \cD_\ell(\o) \big)\big| & = | \wt{B}\, \backslash \,  \cD_\ell(\o)| - \big| \big(\wt{B}\, \backslash \,\cD_\ell(\o)\big) \cap \check{B} \big|
\\[1ex]
&\hspace{-1.5ex}  \stackrel{(\ref{4.30})}{\le} \big| \wt{B}\, \backslash \, \big( \cD_\ell(\o)  \cup \cB_\ell (\o)\big) \big| + |\wt{B} \cap \cB_\ell(\o)| -\mbox{\f $\dis\frac{d}{\nu}$} \, \lol + \ov{c} \,\lol^{1 -  \frac{\chi '}{2}}
\\[1ex]
&\hspace{-3.8ex} \stackrel{(\ref{3.22}),(\ref{2.11})}{\le}  \mbox{\f $\dis\frac{d}{\nu}$} \;\lol + \lol^{\beta'} + 2^d \,\gamma^d_1 \,(\log \log \ell)^d \,\lol^{1- \frac{\kappa}{d}} 
\\
& \qquad \quad - \mbox{\f $\dis\frac{d}{\nu}$} \; \lol +  \ov{c} \,\lol^{1 -  \frac{\chi '}{2}}
\\
& \le \lol^{\wh{\mu}}
\\
&\hspace{-3cm} \mbox{and a similar inequality with $\wt{B}', \check{B}'$ in place of $B, \check{B}$},
\end{split}
\end{equation}
where we have chosen $\wh{\mu}$ in $(0,1)$ bigger than $\max\{1 - \alpha, \beta', 1 - \kappa/d, 1 - \chi'/2\}$ (see also Remark \ref{rem4.6} below).

\medskip
As mentioned above, the balls $\check{B}$ and $\check{B}'$ have volume at most $\frac{d}{\nu} \, \lol + \lol^{\mu'}$, see (\ref{4.26}), and below (\ref{4.28}). Thus, choosing $\wh{\eta}_0 < \wh{\eta}_1$ in $(\frac{\wh{\mu}}{d}, \frac{1}{d})$, both bigger than $\frac{1}{d} - (1 - \mu')$, we see that for large $\ell$, the balls $\check{B}$ and $\check{B}'$ are contained in the balls $\wh{B}_{\rm int}$ and $\wh{B}'_{\rm int}$ with same respective centers as $\check{B}$ and $\check{B}'$ and radius equal to $\wh{R}_{\rm int} = R_0 \lol^{1/d} + \lol^{\wh{\eta}_0}$. The claim (\ref{4.3}) follows. This concludes the proof of Proposition \ref{prop4.5}.
\end{proof}

\begin{remark}\label{rem4.6} \rm The conditions $\wh{\mu} > 1 - \alpha$ and $\frac{\wh{\mu}}{d} < \wh{\eta}_0 < \wh{\eta}_1$ will be helpful in the proof of Theorem \ref{theo4.1} below when bounding the principal Dirichlet eigenvalues in $\wh{B}$ and $\wh{B}'$ as well as the size of the principal Dirichlet eigenfunctions outside $\wh{B}_{\rm int}$ and $\wh{B}'_{\rm int}$. \hfill $\square$
\end{remark}

We now proceed with the

\medskip\n
{\it Proof of Theorem \ref{theo4.1}:} We know by Proposition \ref{prop4.5} that for large $\ell$ on $T \cap R$ we have $\wt{B}, \wt{B}'$ in $\wt{\cC_\ell}$ (see (\ref{3.21})) with $d_\infty (\wt{B}, \wt{B}') \ge \ell^{3/4}$ contained in $B_\ell$ and concentric balls $\wh{B}_{\rm int}$, $\wh{B}$ and $\wh{B}'_{\rm int}, \wh{B}'$ with radii $\wh{R}_{\rm int}, \wh{R}$ and centers with rational coordinates in the respective central boxes of $\wt{B},\wt{B}'$ so that (\ref{4.23}), (\ref{4.24}) hold. We recall that $\wh{B}_{\rm int}, \wh{B}'_{\rm int}$ are closed balls and $\wh{B}, \wh{B}'$ open balls. We denote by (see below (\ref{4.2}) for the terminology)
\begin{equation}\label{4.33}
\mbox{$D_0, D'_0$ the $L_0$-boxes with same respective centers as $\wt{B}$ and $\wt{B}'$}.
\end{equation}

\n
Our main task is to show that $\lo (\wh{B})$ and $\lo (\wh{B}')$ are not much bigger than $\lo (\wt{B})$ and $\lo (\wt{B}')$ respectively, and that $\varphi_{1,\wt{D}_0,\o}$ and $\varphi_{1,D_0,\o}$ are small outside $\wh{B}$ and $\wh{B}'$ respectively. The first point combined with (\ref{4.24}) will yield (\ref{4.6}), (\ref{4.7}), and the second point will prove (\ref{4.8}).

\medskip
We begin with the first point. With $\wh{\mu}, \wh{\eta}_0 < \wh{\eta}_1$ as in Proposition \ref{prop4.5}, we apply Theorem \ref{theo2.3} with the choices $M = 2c_0$, $r = \lol^{\frac{\wh{\mu}}{d}}$, $R = \frac{1}{\sqrt{d}} \;\lol^{\wh{\eta}_1}$, $U_2 = \wt{B}$ (or $\wt{B}')$, $U_1 = \wh{B}$ (or $\wh{B}'$), $\cA = \wh{B}_{\rm int}$ (or $\wh{B}'_{\rm int}$). We note that by (\ref{4.23}) and the fact that $\frac{1-\alpha}{d} < \frac{\wh{\mu}}{d} < \wh{\eta}_0 < \wh{\eta}_1 < \frac{1}{d}$, for large $\ell$ for any $\o$ in $T \cap R$, the assumptions (\ref{2.5}) - (\ref{2.7}) and (\ref{2.9}), (\ref{2.10}) are fulfilled. We then set
\begin{equation}\label{4.34}
\wh{\eta} = ( \wh{\eta}_1 - \wh{\eta}_0) / 2,
\end{equation}
and find by Theorem \ref{theo2.3}, and the upper bounds on $\lo (\wt{B})$ and $\lo (\wt{B}')$ in (\ref{4.24}), that for large $\ell$ on $T \cap R$ (with the above choices for $R$ and $r$):
\begin{equation}\label{4.35}
\left\{ \begin{array}{l}
\lo (\wh{B}) \le \lo (\wt{B}) + 2c_0 \lol^{-2/d} \,\exp\big\{- c_3 \,\big[\mbox{\f $\dis\frac{R}{4r}$}\big]\big\} \le \lo (\wt{B}) + \exp\{ - \lol^{\wh{\eta}}\}
\\[1ex]
\mbox{and similar inequalities with $\wh{B}', \wt{B}'$ in place of $\wh{B}, \wt{B}$}.
\end{array}\right.
\end{equation}
Taking into account the inclusions $\wh{B} \subseteq D_0 \subseteq \wt{B} \subseteq B_\ell$ and $\wh{B}' \subseteq D'_0 \subseteq \wt{B}' \subseteq B_\ell$ for large $\ell$, the claims (\ref{4.6}), (\ref{4.7}) follow.

\medskip
We now turn to the second point, namely the proof of (\ref{4.8}). Recall that $\wh{\eta}_0 > \wh{\mu}/d > (1-\alpha)/d$. Using Lemma \ref{lem2.2}, we assume from now on that $\ell$ is large enough so that for any $y \in \overline{\cD_\ell(\o)}$, the Brownian motion starting at $y$ enters the obstacle set before moving at $| \cdot |_\infty$-distance $\frac{1}{2} \;\lol^{\wh{\eta}_0}$ with probability at least $\frac{1}{2}$, see (\ref{2.4}):
\begin{equation}\label{4.36}
P_y [H_{Obs_\o} < \tau_{\frac{1}{2} \, \lol^{\wh{\eta}_0}}] \ge \fr, \;\mbox{for all $y \in \overline{\cD_\ell(\o)}$}.
\end{equation}

\n
We prove (\ref{4.8}) in the case of $D_0$ and $\wh{B}$ (the case of $D'_0$ and $\wh{B}'$ is handled in the same fashion). By (\ref{4.23}) for any $x \in \wt{B} \backslash \wh{B}_{\rm int}$ the $| \cdot |_\infty$-ball with center $x$ and volume $2 (\log \ell)^{\wh{\mu}}$ has at least half of its volume occupied by $\wt{B}^c \cup \wh{B}_{\rm int} \cup \cD_\ell (\o)$. Thus, Brownian motion starting at $x$ enters $\wt{B}^c \cup \wh{B}_{\rm int} \cup \ov{\cD_\ell (\o)}$ before exiting the concentric box of double radius with a probability at least $c(d) > 0$. Since $\wh{\mu} / d < \wh{\eta}_0$ and $\ell$ is large, the strong Markov property and (\ref{4.36}) shows that the Brownian motion starting at $x$ exits $\wt{B}_\o \backslash \wh{B}_{\rm int}$ (i.e. enters $\wt{B}^c \cup \wh{B}_{\rm int} \cup Obs_\o)$ before moving at $| \cdot |_\infty$-distance $\lol^{\wh{\eta}_0}$ with a non-degenerate probability, namely:
\begin{equation}\label{4.37}
\mbox{for $x \in \wt{B} \backslash \wh{B}_{\rm int}$, $P_x[T_{\wt{B}_\o \backslash \wh{B}_{\rm int}} < \tau] \ge c_5(d) \in (0,1)$,}
\end{equation}
where we have set $\tau = \tau_{\lol^{\wh{\eta}_0}}$ in the notation of (\ref{2.4}).

\medskip
We now write $\varphi$ as a shorthand for $\varphi_{1,D_0,\o}$, see (\ref{1.14}). Since $\wh{\eta}_0 < \frac{1}{d}$ and $\ell$ is large, using (\ref{4.6}), we can further assume that (see (\ref{1.5}) for notation):
\begin{equation}\label{4.38}
\lo (D_0) < 2 c_0 \lol^{-2/d} < \fr \; \lol^{-2 \wh{\eta}_0} \, \lambda_{-\frac{1}{2} \, \Delta} (B_1) \big(= \fr \, \lambda_{-\frac{1}{2}\, \Delta} (B_{\lol^{\wh{\eta}_0}})\big).
\end{equation}
Then, for any $x \in D_{0,\o}$ (see (\ref{1.4})), with $\varphi(x) > 0$ (or equivalently such that the connected component of $D_{0,\o}$ containing $x$ has a principal Dirichlet eigenvalue of $-\frac{1}{2} \, \Delta$ equal to $\lo (D_0)$, see below (\ref{1.14})), we have by (\ref{1.12}):
\begin{equation}\label{4.39}
\varphi(x) = E_x [ \varphi(X_\tau) \, \exp\{ \lo(D_0) \, \tau\}, \; \tau < T_{D_{0,\o}}].
\end{equation}
Denote by $\tau^k$, $k \ge 0$, the iterates of the stopping time $\tau$, that is
\begin{equation}\label{4.40}
\mbox{$\tau^0 = 0, \tau^1 = \tau$, and $\tau^{k+1} = \tau^k + \tau \circ \theta_{\tau^k}$, for $k \ge 1$}
\end{equation}
(with $(\theta_t)_{t \ge 0}$ the canonical shift, see below (\ref{1.7})). Using the strong Markov property and induction, it then follows that
\begin{equation}\label{4.41}
\varphi(x) = E_x[\varphi(X_{\tau^k}) \, \exp\{\lo(D_0) \, \tau^k\}, \; \tau^k < T_{D_{0,\o}}], \; \mbox{for all $k \ge 1$}.
\end{equation}

\n
When in addition $x \in D_{0,\o} \backslash \wh{B}$ and $k\lol^{\wh{\eta}_0} < \frac{1}{\sqrt{d}} \; \lol^{\wh{\eta}_1} \; (\le d_\infty (x, \wh{B}_{\rm int}))$, we find that, with $\wt{\rm sup}_z$ denoting the supremum over $z$ in $D_{0,\o}$ with $d_\infty (z, \wh{B}_{\rm int}) > \lol^{\wh{\eta}_0}$,
\begin{equation}\label{4.42}
\begin{split}
\varphi(x) & \stackrel{(\ref{4.41})}{\le} \|\varphi \|_\infty \, E_x [e^{\lambda_{1,\o}(D_0) \, \tau^k}, \tau^k < T_{D_{0,\o}}]
\\[1ex]
&\hspace{-3ex}  \stackrel{\rm strong \; Markov}{\le}  \|\varphi \|_\infty \; (\wt{\rm sup}_z \;E_z [e^{\lo (D_0) \, \tau}, \tau < T_{D_{0,\o}}])^k
\\[1ex]
&\hspace{-4ex} \stackrel{\rm Cauchy-Schwarz}{\le} \| \varphi \|_\infty \, (\wt{\rm sup}_z \;E_z [e^{2\lo (D_0) \, \tau}] \;P_z [\tau < T_{D_{0,\o}}])^{k/2}.
\end{split}
\end{equation}
Note that for $z \in D_{0,\o}$ with $d_\infty (z, \wh{B}_{\rm int}) > \lol^{\wh{\eta}_0}$ we have
\begin{equation}\label{4.43}
P_z [\tau < T_{D_{0,\o}}] = P_z [\tau < T_{D_{0,\o} \backslash \wh{B}_{\rm int}}] \stackrel{(\ref{4.37})}{\le} (1 - c_5).
\end{equation}
Thus, coming back to (\ref{4.42}), using translation invariance and scaling for the expectation in the last line of (\ref{4.42}), we find for large $\ell$ (see (\ref{2.4}) for notation):
\begin{equation}\label{4.44}
\begin{split}
\varphi(x) &\le  \|\varphi \|_\infty \, \big(E_0 [e^{2\lambda_{1,\o}(D_0)\lol^{2 \wh{\eta}_0} \, \tau_1}] (1 - c_5)\big)^{k/2}
\\
&\hspace{-3.8ex}  \stackrel{(\ref{1.11}),(\ref{4.38})}{\le} c(d,\nu) \lol^{-1/2} \, \big(E_0 [\exp\{4 c_0 \lol^{2 (\wh{\eta}_0 - \frac{1}{d})} \,\tau_1\}] \,(1-c_5)\big)^{k/2}
\\
&\hspace{-1ex} \stackrel{\wh{\eta}_0 < \frac{1}{d}}{\le} c(d,\nu) \lol^{-1/2} \,\big(1 - \mbox{\f $\dis\frac{c_5}{2}$}\big)^{k/2}, \; \mbox{for $x \in D_{0,\o} \backslash \wh{B}$ and $k < \mbox{\f $\dis\frac{1}{\sqrt{d}} $}\; \lol^{\wh{\eta}_1 - \wh{\eta}_0}$}.
\end{split}
\end{equation}

\n
With $\wh{\eta}$ as in (\ref{4.34}), and $\ell$ being large, we see that
\begin{equation}\label{4.45}
\varphi_{1,D_0,\o} \le e^{- \lol^{\wh{\eta}}} \; \mbox{on} \; D_{0, \o} \backslash \wh{B}.
\end{equation}
A similar bound holds for $\varphi_{1,D'_0,\o}$ on $D'_{0,\o} \backslash \wh{B}'$ so that (\ref{4.8}) is proved. This concludes the proof of Theorem \ref{theo4.1}. \hfill $\square$

\begin{remark}\label{rem4.7} \rm 1) The proof of Theorem \ref{theo4.2} is simpler, but similar to the proof of Theorem \ref{theo4.1}. One shows the statement corresponding to Lemma \ref{lem4.4} with one single box $B$ in $\cC^{\rm int}_\ell$ and the statement corresponding to (\ref{4.15}) i), ii) with $\rho_\ell$ set to $0$ (the item i) already follows from $T \subseteq E \cap F$, see (\ref{3.9}), (\ref{3.13})). The statement corresponding to Proposition \ref{prop4.5} now involves a single box $\wt{B}$ and a single ball $\wh{B}_{\rm int}$, with $\rho_\ell$ set to $0$ in the statement corresponding to (\ref{4.24}). The proof then proceeds as that of Theorem \ref{theo4.1} below Remark \ref{rem4.6}.

\medskip\n
2) As already mentioned, we will mainly use (\ref{4.6}) and (\ref{4.8}) of Theorem \ref{theo4.1} in what follows. The statement (\ref{4.7}) is there for clarity and completeness. \hfill $\square$
\end{remark}

\section{Tuning and resonance control}
\setcounter{equation}{0}

In this section we derive an asymptotic upper bound on the resonance event $R$, see (\ref{4.5}), in the large $\ell$ limit, in terms of a quantity, which measures the deconcentration of the law of $\lo (D_0)$ in a suitably tuned regime of low values, with an additional information on the corresponding eigenfunction $\varphi_{1,D_0,\o}$, see Proposition \ref{prop5.3}. In the next section we will prove deconcentration estimates, which will bound the above quantity,  and lead to a lower bound on the spectral gap. We recall the definition of $D_0$-boxes in (\ref{4.2}). Their side-length is $L_0 = 10 (\lceil R_0   \rceil + 1) (\log \ell)^{1/d}$, see (\ref{4.1}).

\medskip
Our first task in this section is to suitably tune the ``low level'' of $\lo(D_0)$ that is pertinent for our purpose. We first need some notation. For $\ell > 10$ we define
\begin{align}
&\mbox{$\wh{\cC}_\ell$ the collection of $D_0$-boxes included in $B_\ell$}, \label{5.1}
\\[1ex]
& \mbox{$\wh{\cC}_\ell^*$ the sub-collection of $\wh{\cC}_\ell$ consisting of boxes $D_{0,q} \subseteq B_\ell$} \label{5.2}
\\
&\mbox{such that $q \in 20 (\lceil R_0 \rceil + 1) \, \IZ^d$ (see (\ref{4.2}) for notation)}. \nonumber
\end{align}

\n
Thus, for large $\ell$,
\begin{align}
&\mbox{the boxes in $\wh{\cC}_\ell^*$ have mutual $| \cdot |_\infty$-distance at least $(\log \ell)^{1/d}$}, \label{5.3}
\\[2ex]
&\mbox{$\wh{\cC}_\ell$ is covered by $\wh{c}$ translates of $\wh{\cC}_\ell^*$ (the choice $\wh{c} = \{40 (\lceil R_0 \rceil + 1)\}^d$ will do)}. \label{5.4}
\end{align}
Also, given an $L_0$-box $D_0$, we write
\begin{equation}\label{5.5}
\mbox{$D_0^{\rm int}$ for the closed concentric box to $D_0$ with side-length $(2 \lceil R_0 \rceil + 4) \lol^{1/d}$}.
\end{equation}
Note that for large $\ell$, in the terminology of (\ref{4.3}), with $\wh{R}$ as above (\ref{4.7}),
\begin{equation}\label{5.6}
\mbox{any ball with center in the central box of $D_0$ and radius $\wh{R}$ is contained in $D_0^{\rm int}$}.
\end{equation}

\n
To specify the relevant low levels of $\lo (D_0)$ for our purpose, we further pick (a large)
\begin{equation}\label{5.7}
\Gamma > 0.
\end{equation}

\n
We will eventually let $\Gamma$ tend to infinity in the next section.

\medskip
Note that $\IP [\lo (D_0) \le t]$ is a non-decreasing, right-continuous function of $t$ in $\IR_+$, which takes the value $0$ for $t = 0$ (actually, for any $t < \lambda_{-\frac{1}{2}\, \Delta}(D_0))$, which tends to $\IP[\lo (D_0) < \infty] = \IP[D_{0,\o} \not= \emptyset]$ as $t$ tends to infinity (see  (\ref{1.4}) for notation). We know from (\ref{1.6}) and monotone convergence that this last quantity tends to $1$ as $\ell$ (and hence $L_0$) goes to infinity. Thus, for $\ell \ge c_6(d,\nu,a, \Gamma)$ so that $\IP[\lo (D_0) < \infty] > \Gamma / |\wh{\cC}_\ell^*|$, we introduce the following  quantile of the law of $\lod$:
\begin{equation}\label{5.8}
t_\ell (\Gamma) = \inf\{ t \ge 0; \, \IP[\lod \le t] \ge \Gamma \, / \, |\wh{\cC}_\ell^*| \},
\end{equation}
so that writing $t_\ell$ as a shorthand for $t_\ell (\Gamma)$, one has for such $\ell$
\begin{equation}\label{5.9}
\left\{ \begin{array}{rl}
{\rm i)} & \IP[ \lod \le t_\ell] \ge  \Gamma \, / \, |\wh{\cC}_\ell^*| ,
\\[2ex]
{\rm ii)} & \IP[ \lod < t_\ell]  \le  \Gamma \,/ \, |\wh{\cC}_\ell^*| .
\end{array}\right.
\end{equation}
We also record the value
\begin{equation}\label{5.10}
s_\ell = \lambda_{-\frac{1}{2} \, \Delta}(D_0) =  \mbox{\f $\dis\frac{d\, \pi^2}{2}$} \;L^{-2}_0, \;\mbox{for which} \; \IP [\lod < s_\ell] = 0.
\end{equation}
The next lemma shows that for large $\ell$ the events $\{\min_{D_0 \in \wh{\cC}_\ell} \lod \ge t_\ell\}$ and $\{\min_{D_0 \in \wh{\cC}_\ell}$ $\lod \le t_\ell\}$ occur with a probability bounded away from $0$.

\begin{lemma}\label{lem5.1}
\begin{equation}\label{5.11}
\liminf_{\ell \r \infty} \; \IP[\min_{D_0 \in \wh{\cC}_\ell} \, \lod \ge t_\ell ] \ge e^{-\wh{c} \; \Gamma} (\mbox{with $\wh{c}$ from (\ref{5.4}))},
\end{equation}
and for large $\ell$,
\begin{equation}\label{5.12}
\IP [ \min_{D_0 \in \wh{\cC}_\ell} \; \lod > t_\ell ] \le e^{-\Gamma}.
\end{equation}
\end{lemma}

\begin{proof}
We first prove (\ref{5.11}). Note that $\lod$ is a non-decreasing function of $\o$. Thus, by the Harris-FKG inequality, see Theorem 20.4, p.~217 of \cite{LastPenr18}, translation invariance, and (\ref{5.4}), we have for large $\ell$
\begin{equation}\label{5.13}
\begin{array}{l}
\IP [ \min_{D_0 \in \wh{\cC}_\ell} \, \lod \ge t_\ell] \ge \IP [ \min_{D_0 \in \wh{\cC}_\ell^*} \, \lod  \ge t_\ell]^{\wh{c}}\; \stackrel{\rm independence}{=}
\\
\IP [ \lod \ge t_\ell]^{\,\wh{c}\, |\wh{\cC}_\ell^*|} = (1- \IP[ \lod < t_\ell])^{\wh{c}\, |\wh{\cC}_\ell^*|}   \stackrel{(\ref{5.9}) ii)}{\ge}
\\[1ex]
(1 - \Gamma \, / \, | \wh{\cC}_\ell^*| )^{\wh{c}\, |\wh{\cC}_\ell^*|}   \underset{\ell \r \infty}{\longrightarrow} e^{-\wh{c} \,\Gamma}.
\end{array}
\end{equation}
This proves (\ref{5.11}). As for (\ref{5.12}), we note that for large $\ell$,
\begin{equation}\label{5.14}
\begin{array}{l}
\IP [ \min_{D_0 \in \wh{\cC}_\ell} \, \lod > t_\ell] \le \IP [ \min_{D_0 \in \wh{\cC}_\ell^*} \, \lod > t_\ell]\;  \underset{\rm transl.\; inv.}{\stackrel{\rm indep.}{=}}
\\
( 1- \IP [ \lod \le t_\ell])^{|\wh{\cC}_\ell^*|}   \stackrel{(\ref{5.9})\, {\rm i)}}{\le} (1 - \Gamma \, / \, | \wh{\cC}_\ell^*| )^{|\wh{\cC}_\ell^*|}  \le e^{-\Gamma}
\end{array}
\end{equation}

\n
(using $1 - s \le e^{-s}$ for all $s$ and $\Gamma / |\wh{\cC}_\ell^*| \le 1$ in the last step). This shows (\ref{5.12}) and hence  Lemma \ref{lem5.1} is proved.
\end{proof}

In the next lemma we collect some coarse asymptotic information on $t_\ell(\Gamma)$ as $\ell$ goes to infinity. We recall $\chi \in (0,d)$ from (\ref{0.4}) and $\gamma_2 > 0$ from (\ref{3.9}).

\begin{lemma}\label{lem5.2}
Given any $\Gamma > 0$, then for large $\ell$
\begin{equation}\label{5.15}
c_0 \lol^{-2/d} - \lol^{-(2 + \chi)/d} \le t_\ell (\Gamma) \le c_0 \lol^{-2/d} + 2 \gamma_2 \lol^{-3/d},
\end{equation}
and in particular
\begin{equation}\label{5.16}
t_\ell (\Gamma) \sim c_0 \lol^{-2/d}, \; \mbox{as $\ell \r \infty$}.
\end{equation}
\end{lemma}

\begin{proof}
We only need to prove (\ref{5.15}) since (\ref{5.16}) is an immediate consequence of (\ref{5.15}). We begin with the right inequality in (\ref{5.15}) and argue as follows. By Theorem \ref{theo4.2} (see (\ref{4.9})) and (\ref{3.25}), with probability tending to $1$ as $\ell$ goes to infinity, $\min_{\wh{\cC}_\ell} \lod$ is at most $c_0 \lol^{-2/d} + 2 \gamma_2 \lol^{-3/d}$. However, by (\ref{5.11}), for large $\ell$, the same quantity is at least $t_\ell (\Gamma)$ with probability bigger or equal to $e^{-\wh{c}\;\Gamma} / 2$. The right inequality in (\ref{5.15}) thus holds for large $\ell$. As for the left inequality in (\ref{5.15}), we note that by (\ref{0.4}), with probability tending to $1$ as $\ell$ goes to infinity, $\min_{D_0 \in \wh{\cC}_\ell} \lod \ge \lob \ge c_0 \lol^{-2/d} - \lol^{-(2 + \chi)/d}$, whereas by (\ref{5.12}), with probability at least $1 - e^{-\Gamma}$ when $\ell$ is large, $\min_{D_0 \in \wh{\cC}_\ell} \lod \le t_\ell (\Gamma)$. The left inequality of (\ref{5.15}) follows. This completes the proof of Lemma \ref{lem5.2}.
\end{proof}

We now come to the main result of this section. It provides an asymptotic upper bound on the probability of the resonance event $R$, see (\ref{4.5}), in terms of the quantity $\Sigma$ in (\ref{5.19}) below. It will play a central role in the ``deconcentration approach'' to the lower bound on the spectral gap developed in the next section. Recall $\wh{\eta} \in (0, \frac{1}{d})$ in Theorem \ref{theo4.1}, $\rho_\ell = \sigma \lol^{-(1+2/d)}$ from (\ref{4.5}), and $\delta_\ell = \lol^{-(2+ 2/d)}$ from (\ref{3.9}). We now set
\begin{equation}\label{5.17}
\ve_\ell = \rho_\ell + \delta_\ell + e^{-\lol^{\wh{\eta}}}.
\end{equation}
We also recall the notation (\ref{5.4}) for $\wh{c}$, (\ref{5.5}) for $D^{int}_0$, $s_\ell$ from (\ref{5.8}) and remind that $t_\ell$ is a shorthand for $t_\ell(\Gamma)$, see below (\ref{5.8}).

\begin{proposition}\label{prop5.3}
For any $\Gamma > 0$ one has
\begin{align}
&\limsup\limits_{\ell \r \infty} \; \IP[R] \le e^{-\Gamma}  + \wh{c} \, \Gamma \,\Sigma + \Sigma^2, \; \mbox{where}\label{5.18}
\\[1ex]
& \Sigma =  \limsup\limits_{\ell \r \infty}  \;|\wh{\cC}_\ell | \; \sup\limits_{t \in [s_\ell, t_\ell]} \; \IP[\lod \in [t, t + \ve_\ell], \; \varphi_{1,D_0,\o} \le e^{-\lol^{\wh{\eta}}} \; \mbox{on} \; D_0 \backslash D_0^{\rm int}]. \label{5.19}
\end{align}
\end{proposition}

\begin{proof}
By Theorem \ref{theo4.1} and the observation (\ref{5.6}), we see that for large $\ell$ one has the inclusion (with hopefully obvious notation):
\begin{equation}\label{5.20}
\begin{split}
T \cap R \subseteq  &\; \mbox{$\{$there are $D_0,D'_0$ in $\wh{\cC}_\ell$ such that $d_\infty(D_0, D'_0) \ge \ell^{3/4}$}
\\
&\; \mbox{and for $D = D_0$ and $D= D'_0$ one has $\lob \le \lo (D) \le \lob + \ve_\ell$}
\\
&\; \mbox{and $\varphi_{1,D,\o} \le e^{-\lol^{\wh{\eta}}}$ on $D \backslash D^{\rm int}\}$}.
\end{split}
\end{equation}
Comparing $\min_{D''_0 \in \wh{\cC}_\ell} \; \lo (D''_0)$ with $t_\ell \, (= t_\ell (\Gamma)$), we see that for large $\ell$ one has by (\ref{5.12})
\begin{equation}\label{5.21}
\begin{array}{l}
\mbox{$\IP[T \cap R]\le e^{-\Gamma} + \IP [\min_{D''_0 \in \wh{\cC}_\ell} \,\lo (D''_0) \le t_\ell$ and there are $D_0,D'_0$ in $\wh{\cC}_\ell$ with}
\\
\mbox{$d_\infty (D_0, D'_0) \ge \ell^{3/4}$, such that for $D = D_0$ and $D = D'_0, \lob \le  \lo (D) \le$}
\\
\mbox{$\lob + \ve_\ell$ and $\varphi_{1,D,\o} \le e^{-\lol^{\wh{\eta}}} \; \mbox{on} \; D \, \backslash \, D^{\rm int}]$}.
\end{array}
\end{equation}
Thus, splitting between the case when $\min_{D''_0 \in \wh{\cC}_\ell} \, \lo (D''_0)$ is strictly smaller, or is equal to $t_\ell$, we now find that for large $\ell$
\begin{equation}\label{5.22}
\begin{array}{l}
\mbox{$\IP[T \cap R]\le e^{-\Gamma} + \IP[$there are $D''$ and $D$ in $\wh{\cC}_\ell$ with $d_\infty (D'', D) \ge \frac{1}{4} \;\ell^{3/4}$}
\\
\mbox{such that $\lo (D'') < t_\ell$ and $\lo (D'') \le \lo (D) \le \lo (D'') + \ve_\ell$, and}
\\
\mbox{$\varphi_{1,D,\o} \le e^{-\lol^{\wh{\eta}}}$ on $D \backslash D^{\rm int}] + \IP[$there are $D$ and $D'$ in $\wh{\cC}_\ell$, with}
\\
\mbox{$d_\infty (D,D') \ge \ell^{3/4}$ such that $\lo(D) \in [t_\ell, t_\ell + \ve_\ell]$, $\varphi_{1,D,\o} \le e^{-\lol^{\wh{\eta}}}$}
\\
\mbox{on $D \backslash D^{\rm int}$, and $\lo(D') \in [t_\ell, t_\ell + \ve_\ell]$, $\varphi_{1,D',\o} \le e^{-\lol^{\wh{\eta}}}$ on $D' \backslash D'^{\rm int}]$}.
\end{array}
\end{equation}

\n
Using the independence of random variables corresponding to $D_0$-boxes at mutual distance bigger than $2a$, one finds with a union bound, independent variables $\o_1$ and $\o_2$, and hopefully obvious notation
\begin{equation}\label{5.23}
\begin{array}{l}
\IP[T \cap R]\le 
\\[0.5ex]
\mbox{$e^{-\Gamma} + | \wh{\cC}_\ell |^2 \,\IP \otimes \IP [\lambda_{1,\o_1} (D'') < t_\ell, \lambda_{1,\o_2}(D) \in [\lambda_{1,\o_1}(D''), {\lo}_1(D'') + \ve_\ell]$ and}
\\[0.5ex]
\mbox{$\varphi_{1,D,\o_2} \le e^{-\lol^{\wh{\eta}}}$ on $D \backslash D^{\rm int}] + | \wh{\cC}_\ell |^2 \;\IP[\lo (D) \in [t_\ell, t_\ell + \ve_\ell]$ and}
\\[0.5ex]
\mbox{$\varphi_{1,D,\o} \le e^{-\lol^{\wh{\eta}}}$ on $D \backslash D^{\rm int}]^2$}.
\end{array}
\end{equation}

\medskip\n
We then bound the product probability in the right member of (\ref{5.23}) with the help of (\ref{5.9}) ii), (\ref{5.4}), as well as (\ref{5.10}), and find that for large $\ell$
\begin{equation}\label{5.24}
\begin{array}{l}
\IP[T \cap R]\le 
\\[0.5ex]
\mbox{$e^{-\Gamma} + \wh{c} \; \Gamma | \wh{\cC}_\ell | \; \sup\limits_{s_\ell \le t \le t_\ell} \, \IP[\lo(D) \in [t,t + \ve_\ell], \varphi_{1,D,\o} \le e^{-\lol^{\wh{\eta}}}$ on $D \backslash D^{\rm int}] \;+$}
\\
\mbox{$| \wh{\cC}_\ell |^2  \, \IP[\lo(D) \in [t_\ell,t_\ell + \ve_\ell], \varphi_{1,D,\o} \le e^{-\lol^{\wh{\eta}}}$ on $D \backslash D^{\rm int}]^2$}.
\end{array}
\end{equation}

\medskip\n
Since $\lim_\ell \,\IP[T] = 0$ by (\ref{3.25}), the claim (\ref{5.18}) now follows from the above inequality and the definition of $\Sigma$ in (\ref{5.19}). This proves Proposition \ref{prop5.3}.
\end{proof}

\section{Lower bound on the spectral gap via deconcentration}
\setcounter{equation}{0}

In this section we prove the main asymptotic lower bound on the spectral gap in Theorem \ref{theo6.1}. The scale $\lol^{-(1 + 2/d)}$ that appears in Theorem \ref{theo6.1} is expected to capture the correct size of the spectral gap, see Remark \ref{rem6.5} 1) at the end of the section. The main ingredient in the proof of Theorem \ref{theo6.1} lies in the deconcentration estimates shown in Theorem \ref{theo6.2}. Whereas the results of the previous sections can be adapted with the techniques of Chapter 4 of \cite{Szni98a} to the case of soft obstacles, see (\ref{1.18}),  the proof of Theorem \ref{theo6.2} uses in a substantial manner the hard sphere obstacles considered here.

\medskip
The main result is 

\begin{theorem}\label{theo6.1} (Lower bound on the spectral gap)
\begin{equation}\label{6.1}
\begin{array}{l}
\lim\limits_{\sigma \r 0} \; \limsup\limits_{\ell \r \infty} \; \IP [R] = 0, \;\mbox{where}
\\
\mbox{$R = \{\o \in \Omega; \lob < \infty$ and $\llob < \lob + \sigma \lol^{-(1 + 2/d)}\}$,}
\\
\mbox{see (\ref{4.5})}.
\end{array}
\end{equation}
\end{theorem}

The main tool in proving Theorem \ref{theo6.1} are the following deconcentration estimates, which take place in a large deviation regime of low values for $\lod$. We recall $D_0$ from (\ref{4.2}), $D_0^{\rm int}$ from (\ref{5.5}), $\wh{\eta} \in (0, \frac{1}{d})$ from Theorem \ref{theo4.1}, $\ve_\ell = \sigma\lol^{-(1 + 2/d)} + \lol^{-(2 + 2/d)} + e^{-\lol^{\wh{\eta}}}$ from (\ref{5.17}), as well as $s_\ell$ and $t_\ell\; (= t_\ell(\Gamma))$ from (\ref{5.10}) and (\ref{5.8}).

\begin{theorem}\label{theo6.2} (Deconcentration estimates)

\n
There is a $K > 0$ such that for every $\Gamma > 0$ and $m \ge 1$, there is a $\sigma_0 > 0$ such that for all $\sigma \in (0,\sigma_0)$, for large $\ell$, for all $t \in [s_\ell, t_\ell]$, setting $J = [t,t + \ve_\ell]$, there are pairwise disjoint compact intervals $J_1,\dots,J_m$ in $(0,t_\ell)$ such that
\begin{equation}\label{6.2}
\begin{array}{l}
\mbox{$\IP [\lod \in J$ and $\varphi_{1,D_0,\o} \le e^{- \lol^{\wh{\eta}}}$ on $D_0 \backslash D_0^{\rm int}] \le K \, \IP [\lod \in J_i]$,}
\\
\mbox{for any $1 \le i \le m$}.
\end{array}
\end{equation}
(We actually prove Theorem \ref{theo6.2} with $K = e^{2 d \nu (1 + 2a)^d}$, see (\ref{6.19})).
\end{theorem}

We first explain how Theorem \ref{theo6.1} follows from Theorem \ref{theo6.2}.

\bigskip\n
{\it Proof of Theorem \ref{theo6.1} based on Theorem \ref{theo6.2}:}  We consider $\Gamma > 0$, $m = \lceil K \, \Gamma^3\rceil$, and a corresponding $\sigma_0 > 0$ as in Theorem \ref{theo6.2}. Then, for all $\sigma \in (0, \sigma_0)$, for large $\ell$, for all $t \in [s_\ell, t_\ell]$, there are pairwise disjoint compact intervals $J_1,\dots, J_m$ in $(0,t_\ell)$ such that (\ref{6.2}) holds. Adding these $m$ inequalities, we see that the left member of (\ref{6.2}) is at most 
\begin{equation*}
\mbox{\f $\dis\frac{K}{m}$}\; \IP[\lod < t_\ell] \stackrel{(\ref{5.9}) {\rm ii)}}{\le}\mbox{\f $\dis\frac{K}{m}$} \; \Gamma / | \wh{\cC}_\ell^*| \stackrel{(\ref{5.4}), \,\rm choice \; of\; m}{\le} \wh{c} / (\Gamma^2 \, |\wh{\cC}_\ell |).
\end{equation*}
Thus, coming back to (\ref{5.19}), we see that for all $\sigma < \sigma_0$,
\begin{equation}\label{6.3}
\Sigma \le \wh{c} \, / \, \Gamma^2.
\end{equation}
It then follows from (\ref{5.18}) that for all $\sigma < \sigma_0$ (which depends on $\Gamma$)
\begin{equation}\label{6.4}
\limsup\limits_{\ell \r \infty} \; \IP [R] \le e^{-\Gamma} + \wh{c}\,^2 \, / \, \Gamma + \wh{c}\,^2 \, / \, \Gamma^4.
\end{equation}
Letting $\Gamma$ tend to infinity, (\ref{6.1}) follows. This proves Theorem \ref{theo6.1}. \hfill $\square$

\medskip
We will now turn to the proof of Theorem \ref{theo6.2}. To establish the deconcentration estimate in a large deviation regime of low values of $\lod$ corresponding to (\ref{6.2}), we will bring into play transformations of the obstacle configurations which do not change too much probabilities (see the multiplicative factor $K$ in (\ref{6.2})), and induce a controlled decrease of $\lod$. In this features lies a difficulty. Whereas increasing $\lod$ is not difficult (for instance by inserting an additional obstacle in a spot where the principal Dirichlet eigenfunction $\varphi_{1,D_0,\o}$ is not too small, see Theorem 2.3, p.~109 of \cite{Szni98a}), decreasing $\lod$ in  a controlled (to have many disjoint intervals $J_i, 1 \le i \le m$, in (\ref{6.2})) and probabilistically economical fashion is a more delicate endeavor. To this end, we will exploit the effect of a ``gentle expansion'' of the Poisson cloud $\o$ to lower eigenvalues. The constraint on $\varphi_{1,D_0,\o}$ in the left member of (\ref{6.2}) will ensure a ``proper centering of the underlying clearing'' in the box (so that it does not get damaged by the expansion). We also refer to Lemma 3.3 of \cite{DingFukuSunXu21} for other kind of transformations, which however do not seem adequate for the task of proving (\ref{6.2}).

\begin{samepage}
\bigskip\n
{\it Proof of Theorem \ref{theo6.2}:} Throughout the proof, using translation invariance, without loss of generality, we let $D_0$ stand for the $L_0$-box centered at the origin, i.e.~corresponding to $q=0$ in (\ref{4.2}). We then consider $\Gamma > 0$, $m \ge 1$ and $\sigma > 0$ to be later chosen small, see (\ref{6.49}). We further introduce the ``expansion ratio''
\begin{equation}\label{6.5}
\mbox{$\lambda = e^{u/ |D_0|}$ with $u \in (0,1)$ \quad $\big($note that $\lambda^d \le 2$ since $\mbox{\f $\dis\frac{d}{|D_0|}$} \stackrel{(\ref{4.1})}{\le} \mbox{\f $\dis\frac{d}{10^d}$} < \log 2\big)$},
\end{equation}
and the homothety centered at the origin of ratio $\lambda$:
\begin{equation}\label{6.6}
h(x) = \lambda x, \; \mbox{for $x \in \IR^d$}.
\end{equation}
Given an $\o = \sum_i \, \delta_{x_i}$ in $\Omega$, we write
\begin{equation}\label{6.7}
\mbox{$\wt{\o} = h \, \circ \, \o = \sum\limits_i \, \delta_{h(x_i)} \in \Omega$, for the point measure image of $\o$ under $h$},
\end{equation}
so that
\begin{equation}\label{6.8}
\begin{array}{l}
\mbox{$\wt{\o}$ under $\IP$ is distributed as $\wt{\IP}$, where $\wt{\IP}$ stands for the law on $\Omega$ of a Poisson}
\\
\mbox{point process on $\IR^d$ with intensity $\frac{\nu}{\lambda^d}$}.
\end{array}
\end{equation}
\end{samepage}

\n
We need some further notation. When $U$ is a bounded open set in $\IR^d$ and $\o \in \Omega$, we write (see (\ref{1.15}))
\begin{equation}\label{6.9}
\wt{\lambda}_{1,\o} (U) = \lambda_{- \frac{1}{2} \, \Delta} \,\big(U \, \backslash \, \bigcup_{x \in \o} B(x,\lambda \, a)\big)
\end{equation}
(i.e.~the principal Dirichlet eigenvalue of $- \frac{1}{2} \, \Delta$ in $U \, \backslash \, \bigcup_{x \in \o} B(x,\lambda \, a)$), and
\begin{equation}\label{6.10}
\begin{array}{l}
\mbox{$\wt{\varphi}_{1,U,\o}$ for the corresponding principal Dirichlet eigenfunction}
\\[0.5ex]
\mbox{(defined similarly to (\ref{1.14}) and below (\ref{1.14}) with $a$ replaced by $\lambda \,a$)}.
\end{array}
\end{equation}
Then, by Brownian scaling, for $U$ bounded open set and $\o \in \Omega$, one has
\begin{align}
 \wt{\lambda}_{1,\wt{\o}} (\lambda U) & = \; \mbox{\f $\dis\frac{1}{\lambda^2}$} \;\lo (U), \; \mbox{and} \label{6.11}
 \\[1ex]
 \wt{\varphi}_{1,\lambda U,\wt{\o}} (\cdot) & = \mbox{\f $\dis\frac{1}{\lambda^{d/2}}$} \; \varphi_{1,U,\o} \, \big(\mbox{\f $\dis\frac{\cdot}{\lambda}$}\big) \; \;\mbox{(both sides identically vanish if $U_\o = \emptyset$)}. \label{6.12}
\end{align}

\n
Of central interest for us is the choice $U = D_0$. By the observation below (\ref{1.10}), $\wt{\varphi}_{1,\lambda D_0,\o}(\cdot)$ is a random continuous function on $\IR^d$, which vanishes outside $\lambda D_0 \backslash \bigcup_{x \in \o} B(x,\lambda\,a)$. Letting $D^a_0$ stand for the open $a$-neighborhood of $D_0$, we see that 
\begin{equation}\label{6.13}
\begin{array}{l}
\mbox{$\wt{\lambda}_{1,\o} (\lambda D_0)$ and $\wt{\varphi}_{1, \lambda D_0, \o}(\cdot)$ are measurable with respect to the $\sigma$-algebra $\cG_{\lambda D_0^a}$} 
\\[0.5ex]
\mbox{on $\Omega$ generated by the random variables $\o(C)$ with $C$ Borel subset of $\lambda D^a_0$}.
\end{array}
\end{equation}

\n
The law $\wt{\IP}_{\lambda D^a_0}$ of the restriction of $\wt{\IP}$ to the $\sigma$-algebra $\cG_{\lambda D^a_0}$ is absolutely continuous with respect to the corresponding restriction $\IP_{\lambda D^a_0}$ of $\IP$ to $\cG_{\lambda D^a_0}$. Indeed, they respectively correspond to a Poisson point process of intensity $\frac{\nu}{\lambda^d}$ and $\nu$ in the box $\lambda D^a_0$ and one has
\begin{equation}\label{6.14}
\begin{split}
\wt{\IP}_{\lambda D^a_0} & = \exp\{(\log \lambda^{-d}) \, \o(\lambda D_0^a) + \nu\, |\lambda D_0^a| - \nu \, \lambda^{-d}\, |\lambda D_0^a|\}\, \IP_{\lambda D^a_0}
\\
& = \exp\{ \nu \,(\lambda^d - 1) \, |D^a_0|\} \; e^{-d \frac{u}{|D_0|} \,\o (\lambda D^a_0)} \; \IP_{\lambda D^a_0} .
\end{split}
\end{equation}
We write as a shorthand notation
\begin{equation}\label{6.15}
\eta_\ell = e^{-\lol^{\wh{\eta}}},
\end{equation}
so that $\ve_\ell = \sigma \lol^{-(1+2/d)} + \lol^{-(2 + 2/d)} + \eta_\ell$, see (\ref{5.17}).

\medskip
We assume from now on $\ell \ge c_6(d,\nu, a, \Gamma)$, see above (\ref{5.8}), so that $t_\ell (= t_\ell (\Gamma))$ is defined. Then for $t \in [s_\ell, t_\ell]$ (see (\ref{5.10}) for notation) and $J = [t, t + \ve_\ell]$, we consider the event that appears in the left member of (\ref{6.2}), namely
\begin{equation}\label{6.16}
\mbox{$A = \{\lod \in J$ and $\varphi_{1,D_0,\o} \le \eta_\ell$ on $D_0 \, \backslash \, D_0^{\rm int}\}$}.
\end{equation}

\n
Using the scaling identities (\ref{6.11}), (\ref{6.12}), and the density formula (\ref{6.14}), we find that
\begin{equation}\label{6.17}
\begin{split}
\IP[A]  \stackrel{(\ref{6.11}), (\ref{6.12})}{=} &\;\IP [\wt{\lambda}_{1,\wt{\o}}  (\lambda D_0) \in \mbox{\f $\dis\frac{J}{\lambda^2}$} \; \mbox{and} \; \wt{\varphi}_{1, \lambda D_0, \wt{\o}} \le \eta_\ell / \lambda^{d/2} \; \mbox{on} \; \lambda (D_0 \backslash D_0^{\rm int})]
\\[0.5ex]
  \stackrel{(\ref{6.8})}{=} \quad\;  &\;\wt{\IP} [\wt{\lambda}_{1,\o} (\lambda D_0) \in  \mbox{\f $\dis\frac{J}{\lambda^2}$}  \; \mbox{and} \; \wt{\varphi}_{1, \lambda D_0, \o} \le \eta_\ell/ \lambda^{d/2} \;\mbox{on} \;  \lambda (D_0 \backslash D_0^{\rm int})]
\\[0.5ex]
 \stackrel{(\ref{6.13}), (\ref{6.14})}{=} &\exp\{ \nu (\lambda^d -1 ) \, |D^a_0| \} \; \IE\, \big[ \exp\big\{  \mbox{\f $-\dis\frac{d u}{|D_0|}$}  \;\o(\lambda D^a_0)\big\}, \; \wt{\lambda}_{1,\o} (\lambda D_0) \in  \mbox{\f $\dis\frac{J}{\lambda^2}$} \; \mbox{and}
\\[0.5ex]
& \wt{\varphi}_{1,\lambda D_0, \o} \le \eta_\ell / \lambda^{d/2} \; \mbox{on} \; \lambda \, (D_0 \backslash D_0^{\rm int})]
\\
\le \quad\; \;\; &K \, \IP[\wt{A}],
\end{split}
\end{equation}
where we have set
\begin{align}
\wt{A} & = \big\{\wt{\lambda}_{1,\o} (\lambda D_0) \in   \mbox{\f $\dis\frac{J}{\lambda^2}$} \; \mbox{and} \; \wt{\varphi}_{1, \lambda D_0, \o} \le \eta_\ell/ \lambda^{d/2} \;\mbox{on} \;  \lambda (D_0 \backslash D_0^{\rm int})\big\} \; \mbox{and} \label{6.18}
\\[1ex]
K & = e^{2d \nu(1 + 2a)^d} \; \mbox{(}\ge \exp\{ \nu\,(\lambda^d - 1)\, |D_0^a|\} \;\mbox{since} \; \lambda^d - 1 = e^{du / |D_0|} - 1 \le 2 du / |D_0| \label{6.19}
\\
&\qquad \qquad \quad \;\;\, \mbox{by (\ref{6.5}) and the inequality $e^s - 1 \le e^s\,s$ for $s > 0$)}. \nonumber
\end{align}

\n
We also note that just as in (\ref{1.11}) (replacing $a$ by $\lambda a$ in the proof of (\ref{1.11})), with the same dimension dependent constant $c_1$, one has
\begin{equation}\label{6.20}
\| \wt{\varphi}_{1,\lambda D_0,\o}\|_\infty \le c_1 \, \wt{\lambda}_{1,\o} (\lambda D_0)^{d/4}, \; \mbox{for all $\o \in \Omega$}
\end{equation}
(one could also use the scaling identities (\ref{6.11}), (\ref{6.12}) and (\ref{1.11}) to infer (\ref{6.20})).

\medskip
We now wish to make use of the inequality (\ref{6.17}). For this purpose we want to compare $\wt{\lambda}_{1,\o} (\lambda D_0)$ that appears in the event $\wt{A}$ in (\ref{6.18}) with $\lod$. The next lemma provides an upper bound on $\lod$ in terms of $\wt{\lambda}_{1,\o} (\lambda D_0)$, when we know that $\wt{\varphi}_{1,\lambda D_0,\o}$ is small on $D^c_0$.

\begin{lemma}\label{lem6.3}
For $\o \in \Omega$, if $t_1 < 1 \, /\, (4 \,|D_0|^{1/2})$ is such that
\begin{equation}\label{6.21}
\wt{\varphi}_{1,\lambda D_0,\o} \le t_1 \; \mbox{on} \; D_0^c.
\end{equation}
then
\begin{equation}\label{6.22}
\lod \le \wt{\lambda}_{1,\o} (\lambda D_0) (1 - 4 t_1 \, |D_0|^{1/2})^{-1}.
\end{equation}
\end{lemma}

\begin{proof}
Without loss of generality, we assume that $\wt{\lambda}_{1,\o}(\lambda D_0) < \infty$ and write $\wt{\varphi}$ as a shorthand for $\wt{\varphi}_{1,\lambda D_0,\o}$. As noted above (\ref{6.13}), $\wt{\varphi}$ is a continuous function. If $t > t_1$ is close to $t_1$ and $4t \, |D_0|^{1/2} < 1$, the function $\wt{\psi} = (\wt{\varphi} - t)_+$ is continuous, compactly supported and vanishes on a neighborhood of $D^c_0$. In addition, one has
\begin{equation}\label{6.23}
1 = \dis\int_{\lambda D_0} \wt{\varphi}\,^2  dx \le \dis\int_{\lambda D_0} \wt{\psi}\,^2  + 2 t \wt{\psi} + t^2 dx \underset{\lambda^d \le 2}{\stackrel{\wt{\psi} = 0 \;\rm on\;D^c_0}{\le}}  \dis\int_{D_0} \wt{\psi}\,^2 dx + 2t \dis\int_{D_0} \wt{\psi} \, dx + 2t^2 \,|D_0|.
\end{equation}

\n
Moreover, by the Cauchy-Schwarz inequality and $\int_{D_0} \wt{\psi}^2 dx \le \int_{D_0} \wt{\varphi}^2 dx \le 1$, one also has $\int_{D_0} \wt{\psi} \,dx \le |D_0|^{1/2}$. So, coming back to (\ref{6.23}), we find that
\begin{equation}\label{6.24}
1 \le \dis\int_{D_0} \wt{\psi}\,^2  dx + 2 t  \,|D_0|^{1/2} + 2t^2 \,|D_0| \stackrel{t \,|D_0|^{1/2} < 1}{\le} \dis\int_{D_0} \wt{\psi}\,^2  dx + 4t \, |D_0|^{1/2}.
\end{equation}

\n
Since $4t \, |D_0|^{1/2} < 1$, it follows that $\wt{\psi}$ is not identically $0$. It vanishes on a neighborhood of $D^c_0$ and on $\bigcup_{y \in \o} B(y, \lambda a)$ and is thus compactly supported in $D_{0,\o}$, which is not empty. By construction it also belongs to $H^1(\IR^d)$ and hence to $H^1_0(D_{0,\o})$. Its Dirichlet integral is at most that of $\wt{\varphi}$ (see for instance Corollary 6.18 on p.~153 of \cite{LiebLoss01}) so that
\begin{equation}\label{6.25}
\lod \le \mbox{\f $\dis\frac{\frac{1}{2} \int |\nabla \wt{\psi}|^2 dx}{\int \wt{\psi}^2 dx}$} \le   \mbox{\f $\dis\frac{\frac{1}{2} \int |\nabla \wt{\varphi}|^2 dx}{\int \wt{\psi}^2 dx}$}  \stackrel{(\ref{6.24})}{\le} \wt{\lambda}_{1,\o} (\lambda D_0) \big(1 - 4 t \,|D_0|^{1/2}\big)^{-1}.
\end{equation}
Letting $t$ decrease to $t_1$, we find (\ref{6.22}).
\end{proof}

In the same fashion, one has
\begin{lemma}\label{lem6.4}
For $\o \in \Omega$, if $t_2 < 1 / (4\, |D_0|^{1/2})$ is such that
\begin{equation}\label{6.26}
\mbox{$\varphi_{1,D_0,\o} \le t_2$ on each $B(y,\lambda a), y \in \o$ that intersect $D_0$},
\end{equation}
then
\begin{equation}\label{6.27}
\wt{\lambda}_{1,\o} (\lambda D_0) \le \lod \big(1 - 4 t_2 \,|D_0|^{1/2}\big)^{-1}.
\end{equation}
\end{lemma}

We now want to choose $t_1$ and $t_2$ so that on the event $\wt{A}$ in (\ref{6.18}) the assumptions of Lemmas \ref{lem6.3} and \ref{lem6.4} are fulfilled. In the case of $t_1$ this is straightforward. We assume $\ell$ sufficient large so that (see (\ref{6.15}), (\ref{4.1}), (\ref{4.2})):
\begin{equation}\label{6.28}
\eta_\ell \, 4 \, |D_0|^{1/2}  \big(= e^{-\lol^{\wh{\eta}}} \,4 \{ 10 ( \lceil R_0 \rceil + 1)\}^{d/2} \, \lol^{1/2}\big) < \fr\;,
\end{equation}
and choose
\begin{equation}\label{6.29}
t_1 = \eta_\ell .
\end{equation}

\n
Then, the assumptions of Lemma \ref{lem6.3} hold on the event $\wt{A}$. Indeed, one has $\wt{\varphi}_{1,\lambda D_0,\o} \le \eta_\ell$ on $\lambda (D_0 \backslash D_0^{\rm int}) \supseteq \lambda D_0 \backslash D_0$ since $\lambda^d \le 2$ and $2 D_0^{\rm int} \subseteq D_0$, see (\ref{5.5}), and $\wt{\varphi}_{1,\lambda D_0,\o}$ vanishes outside $\lambda D_0$.

\medskip
Thus, for large $\ell$, we see that for all $u \in (0,1)$ and $t \in [s_\ell, t_\ell]$ on the event $\wt{A}$ in (\ref{6.18}) (recall that $J = [t,t + \ve_\ell])$ one has
\begin{equation}\label{6.30}
\begin{split}
\lod & \le \wt{\lambda}_{1,\o} (\lambda D_0) \big(1 - 4 \eta_\ell \,|D_0|^{1/2}\big)^{-1} \le (t + \ve_\ell) \, \lambda^{-2} \big(1 - 4\eta_\ell \,|D_0|^{1/2}\big)^{-1}
\\
&\hspace{-3ex}\stackrel{(\ref{6.28}), t \le t_\ell}{\le} 2(t_\ell + \ve_\ell) \stackrel{(\ref{5.15}), (\ref{5.17})}{<} 3 c_0 \lol^{-2/d}.
\end{split}
\end{equation}
We then turn to the choice of $t_2$.

\medskip
We consider $\o \in \wt{A}$ (see (\ref{6.18})) and an arbitrary $y \in \o$ such that $B(y, \lambda a)$ intersects $D_0$ and $\varphi = \varphi_{1,D_0,\o}$ is not identically $0$ on $B(y,\lambda a)$. Since $\varphi = 0$ on $B(y,a)$ we consider $x \in B(y, \lambda a) \backslash B(y,a)$ such that $\varphi(x) > 0$. We write $U$ for the connected component of $D_{0,\o}$ containing $x$ and $U'$ for the intersection of $U$ with $\overset{\circ}{B}(y, 100 a)$, the open ball with center $y$ and radius $100a$. With $\ell$ large and (\ref{6.30}) we can assume that (see (\ref{1.15}), (\ref{1.16}) for notation)
\begin{equation}\label{6.31}
\lod = \lambda_{- \frac{1}{2}\,\Delta} (U) < 3 c_0\, \lol^{-2/d} \le \lambda_d(100a)^{-2} \le \lambda_{- \frac{1}{2}\,\Delta} (U').
\end{equation}
Then, by (\ref{1.12}) (in a rather similar fashion to (\ref{4.39})), we have
\begin{equation}\label{6.32}
\begin{split}
\varphi(x) & = E_x [ \varphi (X_{T_{U'}}) \; \exp\{ \lod \,T_{U'}\}, \, T_{\overset{\circ}{B}(y, 100a)} < T_U]
\\
&\hspace{-4ex}\stackrel{(\ref{6.30}), (\ref{1.11})}{\le} c(d) \,(t + \ve_\ell)^{d/4} \, E_x [\exp\{3 c_0 \lol^{-2/d} \,T_{\overset{\circ}{B} (y, 100a)}\}, T_{\overset{\circ}{B}(y, 100a)} < T_U].
\end{split}
\end{equation}
Note that when Brownian motion enters $B(y,a)$, it exits $U$, so that using scaling and translation invariance, we find that
\begin{equation}\label{6.33}
\varphi(x) \le c(d) (t + \ve_\ell)^{d/4} \sup\limits_{1 \le |z| \le \lambda} E_z [\exp\{3 c_0 \,a^2 \lol^{-2/d} \,T_{\overset{\circ}{B}(0, 100)}\},  \,T_{\overset{\circ}{B}(0, 100)} < H_{B(0,1)}].
\end{equation}

\n
To bound the above expectation, we consider the functions $w(z) = |z|^{-b}$ where $b \in (0,d-2]$. The first and second radial derivatives are $\partial_r \,w = -b \,|z|^{-(b + 1)}$ and $\partial^2_r \,w = b(b + 1) \, |z|^{-(b + 2)}$, so that $\frac{1}{2}\, \Delta w= \frac{1}{2} \, \partial^2_r \,w + \frac{(d-1)}{2 |z|} \;\partial_r \,w =  - \frac{1}{2} \;b (d-2-b) \, |z|^{-(b + 2)}$. We then introduce the stopping time $S = T_{\overset{\circ}{B}(0,100)} \wedge H_{B(0,1)}$, so that for $z \in \IR^d$ with $1 < |z| \le \lambda$, under $P_z$
\begin{equation}\label{6.34}
w (X_{t \wedge S}) \, \exp \big\{- \dis\int^{t \wedge S}_0 \; \mbox{\f $\dis\frac{\frac{1}{2} \; \Delta w}{w}$} \;(X_s) \,ds\big\} = w(X_{t \wedge S}) \;\exp\big\{ \dis\int^{t \wedge S}_0 \; \mbox{\f $\dis\frac{ b (d-2-b)}{2 |X_s|^2}$} \; ds\big\}, \; t \ge 0,
\end{equation}

\n
is a uniformly integrable martingale when $b(d-2-b) < \lambda_d \,100^{-2}$ (it is then dominated by the $P_z$-integrable variable $\exp\{ \frac{1}{2} \,b(d-2-b) \, T_{\overset{\circ}{B}(0,100)}\})$. So for such $b$, after $P_z$-integration and letting $t$ tend to infinity, we have
\begin{equation}\label{6.35}
\begin{split}
|z|^{-b} &= E_z \big[ |X_S|^{-b} \; \exp\big\{ \dis\int^S_0 \mbox{\f $\dis\frac{b(d-2-b)}{2 |X_s|^2}$} \;ds \big\}\big]
\\
& \ge 100^{-b} \, E_z \big[ |X_S| = 100, \, \exp\big\{ \dis\int^S_0  \mbox{\f $\dis\frac{b(d-2-b)}{2 |X_s|^2}$} \;ds \big\}\big] + P_z [|X_S| = 1].
\end{split}
\end{equation}

\n
By a classical formula corresponding to $b = d-2$ in the first line of (\ref{6.35}), one also has $P_z [|X_S| = 100] = (1 - |z|^{-(d-2)}) \, / \, (1-100^{-(d-2)})$ (or see (6) on p.~29 of \cite{Durr84}). We now pick $b = b(d) \in (0, d-2)$ such that $b(d-2-b) < \lambda_d \, 100^{-2}$, and infer from (\ref{6.35}) that
\begin{equation}\label{6.36}
\begin{split}
E_z \big[ |X_S| = 100,\; \exp\big\{ \dis\int^S_0 \mbox{\f $\dis\frac{b(d-2-b)}{2 |X_s|^2}$} \;ds \big\}\big] &\le 100^b (|z|^{-b}  - P_z [|X_S| = 1]) 
\\
&\!\!\!\stackrel{|z| \ge 1}{\le} 100^b \,P_z [|X_S| = 100]
\\
&\!\!\! \stackrel{|z| \ge 1}{\le}  c'(d) \, (1 - \lambda^{-(d-2)}) \stackrel{(\ref{6.5})}{\le} c(d) \, u / |D_0|.
\end{split}
\end{equation}
As a result, when $\ell$ is large enough so that $3 c_0 \,a^2 \lol^{-2/d} < \frac{1}{2} \,b( d- 2 - b) \,100^{-2}$, the expectation in (\ref{6.36}) is bigger or equal to the expectation in (\ref{6.33}), and hence
\begin{equation}\label{6.37}
\varphi(x) \le c''(d) (t + \ve_\ell)^{d/4} u/|D_0| \stackrel{(\ref{6.30}),(\ref{4.1})}{\le} c_7 (d,\nu) \, u / |D_0|^{3/2}, \; \mbox{for all $x \in \bigcup_{y \in \o} B(y,\lambda a)$}.
\end{equation}
We can thus choose
\begin{equation}\label{6.38}
t_2 = c_7 \,u/|D_0|^{3/2}.
\end{equation}
Then, for large $\ell$, for all $u \in (0,1)$ and $s_\ell \le t \le t_\ell$ on $\wt{A}$ in (\ref{6.18}), one has by Lemma \ref{lem6.4}
\begin{equation}\label{6.39}
\wt{\lambda}_{1,\o} (\lambda D_0) \le \lod (1 - 4 c_7 \,u/|D_0|)^{-1}.
\end{equation}
We then introduce the notation (recall that $J = [t,t + \ve_\ell]$)
\begin{equation}\label{6.40}
\begin{split}
\wt{\zeta}_{1,\o} & = \log \big(\wt{\lambda}_{1,\o} (\lambda D_0)\big), \; \zeta_{1,\o} = \log \big(\lambda_{1,\o}(D_0) \big),
\\
\zeta_{\min}& = \log t, \; \zeta_{\max} = \log (t + \ve_\ell).
\end{split}
\end{equation}
Collecting (\ref{6.30}) and (\ref{6.39}), we see that for large $\ell$, for all $u \in (0,1)$, $t \in [s_\ell, t_\ell]$, one has on $\wt{A}$ from (\ref{6.18}), with $c_* = 5c_7 (d,\nu)$:
\begin{equation}\label{6.41}
\wt{\zeta}_{1,\o} - c_* \,u / |D_0| \stackrel{(\ref{6.39})}{\le} \zeta_{1,\o} \stackrel{(\ref{6.30})}{\le} \wt{\zeta}_{1,\o} + 5 \eta_\ell \,|D_0|^{1/2}.
\end{equation}

\medskip\n
However, $e^{\wt{\zeta}_{1,\o}} \in J / \lambda^2$ on $\wt{A}$, with $\lambda = e^{u/|D_0|}$, see (\ref{6.5}), and therefore
\begin{equation}\label{6.42}
\zeta_{\min} - (c_* + 2) \, u/|D_0| \le \zeta_{1,\o} \le \zeta_{\max} + 5 \eta_\ell \,|D_0|^{1/2} - 2 u / |D_0|.
\end{equation}
Note also that $\ve_\ell = \sigma \lol^{-(1 + 2/d)} + \lol^{-(2 + 2/d)} + e^{-\lol^{\wh{\eta}}}$, see (\ref{5.17}), so that for large $\ell$ for all $t \in [s_\ell,t_\ell]$ one has
\begin{equation}\label{6.43}
\zeta_{\max} - \zeta_{\min} = \log \big(1 + \mbox{\f $\dis\frac{\ve_\ell}{t}$}\big) \le  \log \big(1 + \mbox{\f $\dis\frac{\ve_\ell}{s_\ell}$}\big)  \stackrel{(\ref{5.10})}{\le} \ov{c} \,(d,\nu) \, \sigma / |D_0|.
\end{equation}

\n
We are going to choose in (\ref{6.49}) below
\begin{equation}\label{6.44}
\sigma_0 > 0 \; \mbox{and} \; u_0 = 0 < u_1 < \dots u_m < 1,
\end{equation}
satisfying (with $\ov{c}$ from (\ref{6.43}) and $c_*$ from (\ref{6.41}))
\begin{equation}\label{6.45}
2 u_{i+1} > \ov{c} \,\sigma_0 + (c_* + 2) \,u_i, \; \mbox{for} \; i = 0,1,\dots, m-1.
\end{equation}
It will then follow by (\ref{6.43}) that for any $\sigma \in (0, \sigma_0)$, for large $\ell$, for all $t \in [s_\ell, t_\ell]$:
\begin{equation}\label{6.46}
\zeta_{\min} - (c_* + 2) \, u_i / |D_0| > \zeta_{\max} + 5 \eta_\ell \, |D_0|^{1/2} - 2 u_{i+1} / |D_0|, \; \mbox{for $0 \le i < m$},
\end{equation}
so that the intervals
\begin{equation}\label{6.47}
J_i = [t \, e^{-(c_* + 2) \,u_i / |D_0|}, (t + \ve_\ell) \ e^{5 \eta_\ell |D_0|^{1/2} - 2u_i / |D_0|}], \; \mbox{for $1 \le i \le m$},
\end{equation}
are pairwise disjoint, included in $(0,t)$. In addition by (\ref{6.42}), with the choice $u_i$ for $u$, one will have $\lambda_{1,\o}(D_0) \in J_i$ on $\wt{A}$, and by (\ref{6.17}) it will follow that
\begin{equation}\label{6.48}
\IP[A] \le K \, \IP[ \lod \in J_i], \; \mbox{for $1 \le i \le m$ (with $K \stackrel{(\ref{6.19})}{=} e^{2 d \nu(1 + 2a)^d}$)}.
\end{equation}
This will prove Theorem \ref{theo6.2}. There remains to choose $\sigma_0 > 0$ and $u_i, 0 \le i \le m$, so that (\ref{6.44}), (\ref{6.45}) hold. We simply set
\begin{equation}\label{6.49}
\left\{ \begin{array}{l}
\mbox{$u_0 = 0$ and $u_{i+1} = \ov{c} \, \sigma_0 + (c_* + 2) \,u_i$, for $0 \le i <  m$ so that}
\\
u_i = \ov{c} \, \sigma_0 (1 + (c_* + 2) + \dots + (c_* + 2)^{i-1}) = \ov{c} \;\sigma_0\; \mbox{\f $\dis\frac{(c_* + 2)^i - 1}{c_* + 1}$}, \;\mbox{for $1 \le i \le m$},
\end{array}\right.
\end{equation}
and choose $\sigma_0$ small enough so that $u_m < 1$, so that (\ref{6.44}), (\ref{6.45}) hold. This completes the proof of Theorem \ref{theo6.2}. \hfill $\square$

\begin{remark}\label{rem6.5} \rm 1) The results in Section 3 of \cite{Szni97b} make it plausible that the scale $\lol^{-(1 + 2/d)}$ captures the correct size of the spectral gap, in the sense that
\begin{equation}\label{6.50}
\lim\limits_{\sigma \r \infty} \; \liminf\limits_{\ell \r \infty} \;\IP[R] = 1.
\end{equation}
As we explain below (\ref{6.52}), the spectral gap in $B_{4 \ell}$ is related to the fluctuations of $\lob$. The results of Section 3 of \cite{Szni97b} on the fluctuations of $\lob$ are written in the context of soft obstacles, but can be adapted to the present set-up. By (3.19) of Corollary 3.4 and (3.2) of \cite{Szni97b}, one knows that there is a $\zeta \in (0,1)$ such that for any $\ve > 0$ one can find $A_\ve > 0$ such that for large $k_0$
\begin{equation}\label{6.51}
\begin{array}{l}
\mbox{for more than two thirds of $k_0 \le k <  k_0 + [k_0^\zeta]$, there is a compact interval}
\\
\mbox{$I_{2^k, \ve} \subseteq (0,\infty)$ such that $\IP[\lo (B_{2^k}) \notin I_{2^k,\ve}] \le \ve$ and $| I_{2^k,\ve}|  \le A_\ve (\log 2^k)^{-(1 + 2/d)}$}.
\end{array}
\end{equation}

\n
(In essence, $I_{2^k,\ve}$, corresponds to $m_{2^k}(u) - m_{2^k}(v)$ in (3.19) of \cite{Szni98a} with the choice $u = 1- \frac{\ve}{2}$, $v = \frac{\ve}{2}$ and $A_\ve$ to $3 \Gamma$ of the same reference).

\medskip
Now, if $\ell = 2^k$, for a $k$ as above (with the corresponding $I_{\ell, \ve}$), one can consider $B_{4\ell}$ and the $7^d$ sub-boxes $B_{\ell,v} = B_\ell + \ell v$, $ v \in \cV = \{-3, \dots, 3\}^d$ of $B_{4 \ell}$, as well as the event $M_{\ell,\ve}$ such that $F_{(4 \ell)}$ in (\ref{3.9}) (with $4 \ell$ in place of $\ell$) occurs, and $\lo (B_{\ell,v}) \in I_{\ell, \ve}$ for each $v \in \cV$. By (\ref{3.10}) for large $k_0$ and $k$ as in (\ref{6.51}), one has
\begin{equation}\label{6.52}
\IP [M_{\ell = 2^k,\ve}] \ge 1 - (7^d + 1) \, \ve.
\end{equation}

Then, on $M_{\ell = 2^k,\ve}$, one of the $B \in \cC_{(4\ell)}$, see (\ref{3.9}), is such that $\lo (B \cap B_{4\ell})$ lies between $\lo (B_{4\ell})$ and $\lo (B_{4 \ell}) +(\log 4\ell)^{-(2 + 2/d)}$, and the same holds true for $\lo (B_{\ell,v})$ if $B_{\ell,v}$ contains $B \cap B_{4 \ell}$. Thus, $k_0$ (and hence $k$) being large; on $M_{\ell = 2^k,\ve}$, we can consider disjoint boxes $B_{\ell,v}$ and $B_{\ell, v'}$ in $B_{4 \ell}$ such that both $\lo (B_{\ell,v})$ and $\lo (B_{\ell, v'})$ lie in $[\lo (B_{4 \ell})$,  $\lo (B_{4 \ell}) + (\log 4 \ell)^{-(2 + 2/d)} + |I_{\ell, \ve}|]$.  Since $B_{\ell,v}$ and $B_{\ell,v'}$ are disjoint, this implies by the min-max principles, see Version 3 of Theorem 12.1, p.~301 of \cite{LiebLoss01}, that the spectral gap $\llo(B_{4 \ell}) - \lo (B_{4 \ell})$ is at most $(\log 4 \ell)^{-(2 + 2/d)} + |I_{\ell, \ve}| \stackrel{(\ref{6.51})}{\le} (\log 4\ell)^{(-2 + 2/d)} + A_\ve \lol^{-(1 + 2/d)}$. This shows that for any $\ve > 0$, for large $k_0$ (with $2^k$ playing the role of $4 \ell$)
\begin{equation}\label{6.53}
\begin{array}{l}
\mbox{for more than half of $k_0 \le k < k_0 + [k_0^\zeta]$, $\IP [\llo(B_{2^k}) - \lo (B_{2^k})  \le$}
\\
\mbox{$(A_\ve + 1) (\log \ell^k)^{-(1 + 2/d)}] \ge 1 - (7^d + 1) \, \ve$}.
\end{array}
\end{equation}
Incidentally, in the above argument ``half'' could be replaced by any number less than $1$. In any case, from (\ref{6.53}) one sees that it is plausible that (\ref{6.50}) holds.

\bigskip\n
2) One can naturally wonder whether the strategy in the proof of Theorem \ref{theo6.2} can be adapted to derive deconcentration estimates in the context of Poissonian soft obstacles, see (\ref{1.18}), and whether the corresponding lower bound on the spectral gap corresponding to Theorem \ref{theo6.1} can be established in this context as well. Incidentally, in the scaling identities corresponding to (\ref{6.11}), (\ref{6.12}), the original bump function $W(\cdot)$ in (\ref{1.18}) would be transformed into $\wt{W}(\cdot) = \frac{1}{\lambda^2} \, W(\frac{\cdot}{\lambda})$ (the case under present consideration formally corresponds to $W(\cdot) = \infty \, 1\{| \cdot | \le a\}$ and $\wt{W} (\cdot) = \infty \, 1\{| \cdot | \le \lambda a\}$).
 \hfill $\square$
\end{remark}

\section{Bose-Einstein condensation}
\setcounter{equation}{0}

In this section we combine the lower bound on the spectral gap obtained in the main Theorem \ref{theo6.1} with the results of Kerner-Pechmann-Spitzer in \cite{KernPechSpit20}. We prove a so-called type-I generalized Bose-Einstein condensation in probability for a model in the spirit of Kac-Luttinger \cite{KacLutt73}, \cite{KacLutt74} consisting of a non-interacting Bose gas among a Poisson cloud of hard obstacles made of closed balls of radius $a > 0$ centered at the points of the cloud, which has intensity $\nu > 0$. The dimension of space is $d \ge 2$. The radius $a$ although fixed can be arbitrarily small, and the complement in $\IR^d$ of the spherical impurities may possibly percolate, see Chapter 4 of \cite{MeesRoy96}. Our main result is Theorem \ref{theo7.1}.

\medskip
Following  \cite{KernPechSpit20}, we consider a {\it thermodynamic limit} in a grand-canonical set-up. Given a fixed particle density (not to be confused with the parameter of Section 2).
\begin{equation}\label{7.1}
\rho > 0,
\end{equation}
we consider the positive sequence $\ell_N, N \ge 1$, indexed by the particle number $N$, which tends to infinity and satisfies (see (\ref{1.5}) for notation)
\begin{equation}\label{7.2}
\rho\, |B_{\ell_N}| = N, \; \mbox{for $N \ge 1$}
\end{equation}
(in the notation of \cite{KernPechSpit20}, $2 \ell_N = L_N$).

\medskip
The events (see (\ref{1.9}), (\ref{1.4}) for notation)
\begin{equation}\label{7.3}
\Omega_N = \{ \o \in \Omega; \lo (B_{\ell_N}) < \infty\} = \{ \o \in \Omega; B_{\ell_N,\o} \not= \emptyset\}, \; N \ge 1
\end{equation}
are non-decreasing, have positive probability, and
\begin{equation}\label{7.4}
\mbox{$\bigcup_{N \ge 1} \, \Omega_N = \Omega_\infty$ (the event of full $\IP$-measure in (\ref{1.6}))}.
\end{equation}

\n
To stay within the framework of \cite{KernPechSpit20}, on $\Omega^c_N$ we replace the variables $\lo(B_{\ell_N}), j \ge 1$, which are all infinite on $\Omega^c_N$, by the ordered sequence of Dirichlet eigenvalues of $-\frac{1}{2}\,\Delta + 1$ in $B_{\ell_N}$ (any positive constant in place of $1$ would do). We thus define the modified single particle eigenvalues (see (\ref{1.9}) for notation)
\begin{equation}\label{7.5}
\ov{\lambda}_{j,\o} (B_{\ell_N}) = \left\{ \begin{array}{l}
\mbox{$\lambda_{j,\o} (B_{\ell_N})$,  when $j \ge 1$ and $\o \in \Omega_N$},
\\[1ex]
\mbox{the $j$-th Dirichlet eigenvalue of $-\frac{1}{2} \, \Delta + 1$ in $B_{\ell_N}$, when $j \ge 1$ and}
\\
\mbox{$\o \notin \Omega_N$}.
\end{array}\right.
\end{equation}
Note that
\begin{equation}\label{7.6}
\mbox{on $\Omega_\infty$, for large $N$, $\lambda_{j,\o} (B_{\ell_N}) = \ov{\lambda}_{j,\o} (B_{\ell_N})$, for all $j \ge 1$},
\end{equation}
i.e.~outside a $\IP$-negligible set, the original sequence of eigenvalues agrees with the modified sequences of eigenvalues, when $N$ is large. We denote by $\beta \in (0,\infty)$ the inverse temperature. Given a chemical potential $\mu \in (-\infty, \ov{\lambda}_{1,\o} (B_{\ell_N})$), the quantity $(e^{\beta(\ov{\lambda}_{j,\o} (B_{\ell_N}) - \mu)} - 1)^{-1}$ is the number of particles occupying the $j$-the eigenstate in the grand-canonical ensemble, see (2.7) of \cite{KernPechSpit20},  and we choose $\mu = \mu^\o_N \in (-\infty, \ov{\lambda}_{1,\o}(B_{\ell_N}))$ so that
\begin{equation}\label{7.7}
\mbox{\f $\dis\frac{1}{|B_{\ell_N}|}$} \; \sum\limits_{j \ge 1} \; \ov{n}\,^{j,\o}_N = \rho, \; \mbox{with} \; \ov{n}\,^{j,\o}_N = 1 \, / \, (e^{\beta(\ov{\lambda}_{j,\o} (B_{\ell_N}) - \mu^\o_N)} - 1).
\end{equation}
In view of (\ref{7.6}) and setting $n^{j,\o}_N = 1 \, / \, (e^{\beta(\lambda_{j,\o} (B_{\ell_N}) - \mu^\o_N)} - 1)$ on $\Omega_N$, we also see that
\begin{equation}\label{7.8}
\mbox{on $\Omega_\infty$, for large $N$, $\ov{n}\,^{j,\o}_N = n^{j,\o}_N$, for all $j \ge 1$},
\end{equation}
(i.e.~the modification (\ref{7.5}) is immaterial on $\Omega_\infty$ for large $N$).

\medskip
We recall some known facts. As $N \r \infty$, the random measures on $\IR_+$
\begin{equation}\label{7.9}
m_{N,\o} = \left\{ \begin{array}{l}
\mbox{\f $\dis\frac{1}{|B_{\ell_N}|}$} \; \sum\limits_{i \ge 1} \; \delta_{\lambda_{i,\o} (B_{\ell_N})}, \; \mbox{for $\o \in \Omega_N$},
\\[2ex]
0, \;\; \mbox{for $\o \in \Omega^c_N$},
\end{array}\right.
\end{equation}
are known to $\IP$-a.s. converge vaguely to a deterministic measure, the {\it density of states}
\begin{equation}\label{7.10}
\begin{array}{l}
\mbox{$m$ on $[0,\infty)$, characterized by its Laplace transform}
\\
\dis\int_{[0,\infty)} e^{-t \lambda} dm(\lambda) = \mbox{\f $\dis\frac{1}{(2 \pi t)^{d/2}}$}\;E^t_{0,0} [\exp \{ - \nu \,|W_t^a|\}], \; \mbox{for $t > 0$},
\end{array}
\end{equation}
where $E^t_{0,0}$ denotes the expectation with respect to the Brownian bridge in time $t$ from $0$ to $0$, and $W_t^a$ is the Wiener sausage of radius $a$ in time $t$, i.e.~the closed $a$-neighborhood of the Brownian bridge trajectory. (The proof is similar to that of Theorem 5.18, p.~99 of \cite{PastFigo92}). In view of (\ref{7.6}) we also see that
\begin{equation}\label{7.11}
\begin{array}{l}
\mbox{$\IP$-a.s., the measures $\ov{m}_{N,\o} = \mbox{\f $\dis\frac{1}{|B_{\ell_N}|}$} \; \dis\sum\limits_{i \ge 1} \; \delta_{\ov{\lambda}_{i,\o}(B_{\ell_N})}$ on $[0,\infty)$ converge vaguely}
\\
\mbox{as $N \r \infty$ to the measure $m$ in (\ref{7.10})}.
\end{array}
\end{equation}
One also knows that $m$ has a so-called Lifshitz tail close to $0$. More precisely, see Corollary 3.5 of \cite{Szni90a} or as in Theorem 10.2, p.~221 of \cite{PastFigo92}, one has (see above (\ref{0.2}) for notation):
\begin{equation}\label{7.12}
m([0,\lambda]) = \exp\big\{ - \nu \o_d (\lambda_d / \lambda)^{d/2} \big(1 + o(1)\big)\big\}, \; \mbox{as $\lambda \r 0$}.
\end{equation}
(The quantity $ \exp\{ - \nu \o_d (\lambda_d / \lambda)^{d/2}\}$ is the probability that the Poisson point process places no point in the open ball of radius $(\lambda_d/\lambda)^{1/2}$ centered at the origin. This ball has a principal Dirichlet eigenvalue for $-\frac{1}{2}\,\Delta$ equal to $\lambda$.) We then introduce the critical density for our system (see (2.6) of \cite{KernPechSpit20}), namely
\begin{equation}\label{7.13}
\rho_c(\beta) = \dis\int^\infty_0 \; \mbox{\f $\dis\frac{1}{e^{\beta \lambda} - 1}$} \; dm(\lambda) \; (< \infty \; \mbox{by (\ref{7.12}))}.
\end{equation}
In our model a {\it generalized Bose-Einstein condensation}, with a macroscopic occupation of an arbitrary small energy band of one-particle states, is known to occur when $\rho > \rho_c(\beta)$, see Theorem 2.5 of \cite{KernPechSpit20} or Theorem 4.1 of \cite{{LenoPastZagr04}}. 
The main result of this section is the following theorem, which pins down the nature of the condensation. It shows a type-I generalized Bose-Einstein condensation in probability, when $\rho > \rho_c(\beta)$:
\begin{theorem}\label{theo7.1}
When $\rho > \rho_c(\beta)$, then as $N$ tends to infinity,
\begin{equation}\label{7.14}
\left\{ \begin{array}{rl}
{\rm i)} & \mbox{$\ov{n}^{1,\o}_N / N$ tends to $\big(\rho - \rho_c(\beta)\big)/ \rho$ in $\IP$-probability, and}
\\[2ex]
{\rm ii)} & \mbox{$\ov{n}^{j,\o} / N$ tends to $0$ in $\IP$-probability, for any $j \ge 2$.}\qquad
\end{array}\right.
\end{equation}
Equivalently, for every fixed $N_1 \ge 1$, as $N \ge N_1$ tends to infinity,
\begin{equation}\label{7.15}
\left\{ \begin{array}{rl}
{\rm i)} & \mbox{$n^{1,\o}_N / N$ tends to $\big(\rho - \rho_c(\beta)\big)/ \rho$ in $\IP_{N_1}$-probability, and}
\\[2ex]
{\rm ii)} & \mbox{$n^{j,\o}_N / N$ tends to $0$ in $\IP_{N_1}$-probability,}
\end{array}\right.
\end{equation}
where $\IP_{N_1}$ stands for the conditional probability $\IP ( \cdot \, | \, \Omega_{N_1})$, see (\ref{7.3}).
\end{theorem}

(When $\rho \le \rho_c(\beta)$, the above quantities in (\ref{7.14}), respectively in (\ref{7.15}), are known to converge to $0$ in $\IP$-probability, respectively $\IP_{N_1}$-probability, as $N \r \infty$, see for instance Theorem 3.2.4, p.~30 of \cite{Pech19}.)

\begin{proof}
The claim (\ref{7.15}) is a simple restatement of (\ref{7.14}) since $n^{j,\o}_N, j \ge 1, N \ge N_1$ coincide with $\ov{n}^{1,\o}_N, j \ge 1, N \ge N_1$ on $\Omega_{N_1}$ and $\lim_{N_1 \r \infty} \IP(\Omega_{N_1}) = 1$, see (\ref{7.4}), (\ref{7.5}), (\ref{7.7}). We only need to prove (\ref{7.14}).
 
 \medskip
 We want to apply Theorem 2.9 of \cite{KernPechSpit20}. We first need to check the Assumptions 2.2 of the above reference. We have $\IP$-a.s., $\ov{\lambda}_{1,\o} (B_{\ell_N}) \r 0$, as $N \r 0$, by (\ref{0.4}) and (\ref{7.6}), and $\IP$-a.s., $\ov{m}_{N,\o}$ converges vaguely to the deterministic measure $m$, as $N \r \infty$, by (\ref{7.11}). This shows that i) and ii) of Assumptions 2.2 of \cite{KernPechSpit20} hold. As we now show
 \begin{equation}\label{7.16}
\IE [ \ov{m}_{N,\o} ([0,\lambda])] \le m([0,\lambda]) \; \mbox{for $\lambda < 1$ and $N \ge 1$}.
\end{equation}

\medskip\n
This will imply that iii) of Assumptions 2.2 of the above reference holds as well. To prove (\ref{7.16}), we first observe that by Dirichlet bracketing and translation invariance, for $N, k \ge 1, \lambda \ge 0$, $\IE[m_{N,\o} ([0,\lambda])] \le \IE[m_{k^d N}([0,\lambda])]$ (we recall that by (\ref{7.2}) $\ell_{k^d N} = k \ell_N$). Then, on a set of full $\IP$-measure $m_{k^d N,\o}$ converges vaguely to $m$ as $k$ goes to infinity (see (\ref{7.9}), (\ref{7.10})), and hence on that set for every point $\lambda$ of continuity of $m([0, \cdot])$, $\lim_k m_{k^d N,\o}([0,\lambda]) = m([0,\lambda])$.  For any such fixed $\lambda$ we also have $m_{k^d N,\o}  ([0,\lambda]) \le m_{k^d N,\o = 0} ([0,\lambda])$ which is deterministic and bounded in $k$. Thus, by dominated convergence, we see that for any $\lambda$ point of continuity of $m ([0,\cdot])$, one has
 \begin{equation}\label{7.17}
\IE [m_{N,\o} ([0,\lambda])] \le m([0,\lambda])
\end{equation}
and this inequality extends to any $\lambda \ge 0$ by right-continuity and monotonicity. Since $ \ov{m}_{N,\o} ([0,\lambda]) = m_{N,\o} ([0,\lambda]) $ for all $\lambda < 1$ and $\o \in \Omega$ by (\ref{7.5}), the claim (\ref{7.16}) follows. 

\medskip
As for  iv) of Assumptions 2.2 in \cite{KernPechSpit20}, by the Lifshitz tail asymptotics (\ref{7.12}), one has, picking $\eta_1 \in (0,1)$ and setting $\wt{C}_1 = \nu \, \o_d \, \lambda_d^{d/2} (1 + \eta_1/2) \stackrel{(\ref{0.3})}{=} d \,c_0^{d/2} (1 + \eta_1 /2)$,
 \begin{equation}\label{7.18}
\lim\limits_N \; N^{1- \eta_1} \; m([0, \{\wt{C}_1 / \log (|B_{\ell_N}|)\}^{2/d}]) = 0.
\end{equation}
This shows that Assumptions 2.2 of  \cite{KernPechSpit20} are satisfied.

\medskip
Now by the lower bound on the spectral gap in Theorem \ref{theo6.1} and (\ref{7.6}), 
 \begin{equation}\label{7.19}
\lim\limits_N \; \IP [\ov{\lambda}_{2,\o} (B_{\ell_N}) - \ov{\lambda}_{1,\o} (B_{\ell_N}) \ge (\log \ell_N)^{-(1 + 2/d + \ve)}] = 1, \; \mbox{for any $\ve > 0$}.
\end{equation}
Together with (\ref{0.4}) (or (\ref{3.10})) and the fact that $\log |B_{\ell_N}| \sim d \log \ell_N$, as $N \r \infty$, this is more than enough to show that for $\wt{C}_2 = d \, c_0^{d/2} (1 + \eta_1 / 4)$ ($< \wt{C}_1$ above), one has
 \begin{equation}\label{7.20}
\lim\limits_N \; \IP [\ov{\lambda}_{2,\o} (B_{\ell_N}) - \ov{\lambda}_{1,\o} (B_{\ell_N}) \ge  N^{-(1 - \eta_1)} \;\mbox{and} \;   \ov{\lambda}_{1,\o} (B_{\ell_N}) \le \{\wt{C}_2 / \log (|B_{\ell_N}|)\}^{2/d}] = 1.
\end{equation}
The assumptions of Theorem 2.9 of \cite{KernPechSpit20} are thus fulfilled (with $c_2 =1$ and $c_3 =1$), and the claim (\ref{7.14}) follows. This concludes the proof of Theorem \ref{theo7.1}.
\end{proof}

\begin{remark}\label{rem7.2} \rm 1) One can wonder whether the above results extend to the context of soft Poissonian obstacles, see (\ref{1.18}). In the case of dimension $1$ we refer to \cite{Pech20} for results in the case when the strength of the soft obstacle tends to infinity with $N$, see Corollary 5.8 of this reference, and a less specific  companion  statement in the case of a fixed strength, see Corollary 5.6 of \cite{Pech20}.

\bigskip\n
2) In the case of hard spherical Poissonian obstacles in $\IR^d, d\ge 2$, in a suitable non-percolative regime for the vacant set, and suitably strong short-range repulsive pair-interactions, we refer to \cite{KernPech21b}, which, among other results, shows the absence of Bose-Einstein condensation into the normalized eigenstates of the Dirichlet Laplacian in $B_{\ell_N,\o}$ (in the notation of (\ref{1.4})). \hfill $\square$
\end{remark}

%\bibliographystyle{plain}
%\bibliography{./literature-emptycode}

\begin{thebibliography}{10}

\bibitem{Adam75}
R.A. Adams.
 {\em Sobolev {S}paces}. Academic Press, New York, 1975.

\bibitem{Astr16}
A.~Astrauskas.
 From extreme values of i.i.d. random fields to extreme eigenvalues of  finite-volume {A}nderson {H}amiltonian.
 {\em Probab. Surv.}, 13:156--244, 2016.

\bibitem{BiskKoni16}
M.~Biskup and W.~K\"onig. Eigenvalue order statistics for random {S}chr\"odinger operators with doubly-exponential tails.
 {\em Commun. Math. Phys.}, 341(1):179--218, 2016.

\bibitem{BrasPhilVeli15}
L.~Brasco, G.~De Philippis, and D.~Velichkov. Faber-{K}rahn inequalities in sharp quantitative form. {\em Duke Math. J.}, 9:1777--1831, 2015.

\bibitem{DingFukuSunXu21}
J.~Ding, R.~Fukushima, R.~Sun, and C.~Xu. Distribution of the random walk conditioned on survival among  quenched {B}ernoulli obstacles.
 {\em Ann. Probab.}, 49(1):206--243, 2021.

\bibitem{DingXu20}
J.~Ding and C.~Xu.
 Localization for random walks among random obstacles in a single Euclidean ball. {\em Commun. Math. Phys.}, 375(2):949--1001, 2020.

\bibitem{DonsVara75b}
M.~Donsker and S.R.S. Varadhan. Asymptotics for the {W}iener sausage. {\em Comm. Pure Appl. Math.}, 28(4):525--565, 1975.

\bibitem{DumiRiveRodrVann21}
H.~Duminil-Copin, A.~Rivera, P.-F. Rodriguez, and H.~Vanneuville. Existence of an unbounded nodal hypersurface for smooth {G}aussian
  fields in dimension $d \ge 3$. {\em Ann. Probab.}, 51(1), 228--276, 2023.

\bibitem{Durr84}
R.~Durrett.
 {\em Brownian motion and martingales in analysis}. Wadsworth, Belmont CA, 1984.

\bibitem{FuscMaggPrat09}
N.~Fusco, F.~Maggi, and A.~Pratelli.
 Stability estimates for certain {F}aber-{K}rahn, isocapacitary and  {C}heeger inequalities. {\em Ann. Scuola Norm. Sup. Pisa Cl. Sci. (5)}, 8:51--71, 2009.

\bibitem{GermKlop13}
F.~Germinet and F.~Klopp.
 Enhanced {W}egner and {M}inami estimates and eigenvalue statistics of  random {A}nderson models at spectral edges.
 {\em Ann. Henri Poincar\'e}, 14(5):1263--1285, 2013.

\bibitem{KacLutt73}
M.~Kac and J.M. Luttinger.
 Bose-{E}instein condensation in the presence of impurities~{I}. {\em J. of Math. Phys.}, 14(11):1626--1628, 1973.

\bibitem{KacLutt74}
M.~Kac and J.M. Luttinger.
 Bose-{E}instein condensation in the presence of impurities {II}. {\em J. of Math. Phys.}, 15(2):183--186, 1974.



\bibitem{KernPech21b}
J.~Kerner and M.~Pechmann.
 Bose-Einstein condensation for particles with repulsive short-range pair interactions in a Poisson random external potential in ${\IR}^d$. {\it Preprint}, arXiv:2110.04587, 2021.


\bibitem{KernPechSpit19b}
J.~Kerner, M.~Pechmann, and W.~Spitzer.
On {B}ose-{E}instein condensation in the {L}uttinger-{S}y model with finite interaction strength.
{\em J. Stat. Phys.}, 174:1346--1371, 2019.


\bibitem{KernPechSpit20}
J.~Kerner, M.~Pechmann, and W.~Spitzer. On a condition for type-I Bose-Einstein condensation in random  potentials in $d$ dimensions.
 {\em J. Math. Pures Appl.}, 143:287--310, 2020.


\bibitem{Koni16}
W.~K\"onig.
 {\em The parabolic {A}nderson model: Random walk in random  potential}. Pathways in Mathematics, Birkh\"auser, 2016.

\bibitem{LastPenr18}
G.~Last and M.~Penrose.
 {\em Lectures on the {P}oisson process. {\rm Institute of  Mathematical Statistics Textbooks}}, volume~7, Cambridge University Press, 2018.

\bibitem{LenoPastZagr04}
O.~Lenoble, L.A. Pastur, and V.A. Zagrebnov. Bose-{E}instein condensation in random potentials. {\em C.R. Physique}, 5:129--142, 2004.

\bibitem{LiebLoss01}
E.~Lieb and M.~Loss.
 {\em Analysis}, volume~14 of {\em Graduate Studies in Mathematics}. Second edition, AMS, 2001.

\bibitem{MeesRoy96}
R.~Meester and R.~Roy.
{\it Continuum Percolation}. Cambridge  Tracts in Mathematics, 1996.

\bibitem{PastFigo92}
L.~Pastur and A.~Figotin.
 {\em Spectra of Random and Almost Periodic Operators}. Springer, New York, 1992.

\bibitem{Pech19}
M.~Pechmann.
{\em Bose-{E}instein condensation in random potentials}.
PhD thesis,\linebreak FernUniversität in Hagen 2019, available at  https://doi.org/10.18445/20190724-144457-2.


\bibitem{Pech20}
M.~Pechmann.
 On {B}ose-{E}instein condensation in one-dimensional noninteracting  {B}ose gases in the presence of soft {P}oisson obstacles.
 {\em J. Stat. Phys.}, 189(3), Paper 42, 2022.

\bibitem{PoisSime22}
J.~Poisat and F.~Simenhaus.
 Localization of a one-dimensional simple random walk among power-law  renewal obstacles. {\em Preprint, {\rm arXiv:2201.05377}}, 2022.

\bibitem{Szni90a}
A.S. Sznitman.
 Lifschitz tail and {W}iener sausage, {I}. {\em J. Funct. Anal.}, 94(2):223--246, 1990.

\bibitem{Szni97b}
A.S. Sznitman.
 Fluctuations of principal eigenvalues and random scales. {\em Comm. Math. Phys.}, 189(2):337--363, 1997.

\bibitem{Szni98a}
A.S. Sznitman.
 {\em Brownian motion, obstacles and random media}. Springer, Berlin, 1998.

\end{thebibliography}

\end{document}